\documentclass[10pt]{article} 

\usepackage[latin1]{inputenc} 
\usepackage{geometry} 
\usepackage{array} 

\usepackage{xargs}
\usepackage{bbm}
\usepackage[all]{xy}
\usepackage{amssymb}
\usepackage{amsmath}
\usepackage{amsthm}
\usepackage{mathrsfs}

\theoremstyle{plain} 
\newtheorem{thm}{Theorem}[section]
\newtheorem{cor}[thm]{Corollary}
\newtheorem{prop}[thm]{Proposition}
\newtheorem{lem}[thm]{Lemma}
\newtheorem*{theoetoile}{Theorem} 

\theoremstyle{definition} 
\newtheorem{definition}[thm]{Definition}

\newtheorem{remarque}[thm]{Remark}
\newtheorem{notation}[thm]{Notation}

\newcommand{\Aa}{\mathscr{A}}
\newcommand{\action}{\curvearrowright}

\newcommand{\AlgLie}{\mathfrak{g}}

\newcommand{\C}{\mathbbm{C}}

\newcommand{\derp}[2][]{\frac{\partial #1}{\partial #2}}

\newcommand{\HH}{H_1} 
\newcommand{\inj}{\hookrightarrow}
\newcommand{\Kk}{\mathcal{K}}
\newcommand{\LieAlg}{\AlgLie}
\newcommand{\Modl}{\mathscr{E}}

\newcommandx*{\ModQHM}[2][1={},2={}]{{M^{#1}_{#2}}}
\newcommandx*{\ModQHMl}[2][1={},2={}]{{\mathscr{M}^{#1}_{#2}}}
\newcommand{\N}{\mathbbm{N}}
\newcommandx*{\ModN}[2][1={},2={}]{{N^{#1}_{#2}}}
\newcommandx*{\ModNN}[2][1={},2={}]{{N^{\dag\,#1}_{#2}}}
\newcommandx*{\ModNl}[2][1={},2={}]{{\mathscr{N}^{#1}_{#2}}}
\newcommandx*{\ModNNl}[2][1={},2={}]{{\mathscr{N}^{\dag\,#1}_{#2}}}
\newcommand{\PiCrG}{{A \rtimes _E \Z}}

\newcommandx*\QHM[2][1={},2={}]{{D^{#1}_{#2}}}
\newcommandx*\QHMl[2][1={},2={}]{{\mathscr{D}^{#1}_{#2}}}
\newcommand{\R}{\mathbbm{R}}
\newcommand{\Ss}{\mathscr{S}}
\newcommand{\Toep}{\mathscr{T}}

\newcommand{\Tt}{\mathcal{T}}
\newcommand{\U}{\mathbb{U}}
\newcommand{\Uu}{\mathcal{U}}
\newcommand{\Z}{\mathbbm{Z}}

\newcommand{\ToepQHM}{\Tt_{\ModQHM}}
\newcommand{\ToepQHMl}{\Toep_{\ModQHMl}}
\newcommand{\kerl}{\mathscr{J}}

\DeclareMathOperator{\Ch}{Ch}
\DeclareMathOperator{\id}{Id}
\DeclareMathOperator{\supp}{Supp}
\DeclareMathOperator{\tr}{Tr}

\renewcommand{\colon}{\mathclose{}\mathpunct{}\mathpunct{:}}

\numberwithin{equation}{section}

\usepackage{hyperref}

\title{$K$-theory, Cyclic Cohomology and Pairings for Quantum~Heisenberg~Manifolds}
\author{Olivier \textsc{Gabriel}}

\begin{document}
\maketitle

\abstract{The $C^*$-algebras called Quantum Heisenberg Manifolds (QHM) were introduced by Rieffel in 1989 as strict deformation quantizations of Heisenberg manifolds. It was later shown that they are also examples of generalized crossed products. In this article, we compute the pairings of $K$-theory and cyclic cohomology on the QHM. Combining these calculations with other results proved elsewhere, we also determine the periodic cyclic homology and cohomology of these algebras, and obtain explicit bases of the periodic cyclic cohomology of the QHM. We further isolate bases of periodic cyclic homology, expressed as Chern characters of the $K$-theory.}

\medbreak

\textsl{Mathematics Subject Classification (2000).} 46L87, 19K99.

\textsl{Keywords.} Quantum Heisenberg manifolds, $K$-theory, cyclic cohomology, Chern-Connes pairings, generalized crossed product, Pimsner algebra, Heisenberg group.

\tableofcontents

\section{Introduction}

Quantum Heisenberg Manifolds (QHM) were introduced by Rieffel in \cite{RieffelDefQuant}. They are a family of $C^*$-algebras $\QHM[c][\mu,\nu]$  indexed by $c \in \Z$ and $\mu,\nu \in \R$. In the following, we denote these algebras $\QHM$ for brevity. In the article \cite{AbadieEE} by Abadie, Eilers and Exel, it was proved that QHM are \emph{generalized crossed products}. One can prove that QHM are also \emph{Pimsner algebras}, a closely related notion which was introduced in the seminal article by Pimsner \cite{PiCrG}. Since this later class has been more widely studied (see for example \cite{IdStrPinzari,DPZ,IdAlgCP,FMR}), let us mention that both notions often coincide and that a common generalization can be found (see \cite{Katsura04}). 

QHM have been closely studied in series of articles by Abadie (see for instance \cite{AbadieVectBundle,AbadieFixedPts,AbadieExel,AbadieTraceRange}) and Chakraborty (see \cite{GeomQHM} and \cite{MetricQHeisenMan}). In their article \cite{SphereNC2}, Connes and Dubois-Violette related QHM to the noncommutative 3-spheres they introduced in \cite{SphereNC1}. More recently, Kang studied these algebras through their Yang-Mills functionals (\cite{LaplaceQHM}) and Kumjian started considering them as groupoids algebras (\cite{QHMKumjian}).

The present article is a refinement of the previous studies of QHM, by means of pairings between cyclic cohomology and $K$-theory. These pairings were first defined by Connes in his article \cite{IHES} of 1985. They have a parity: odd cocycles pair with $K_1$ and even cocycles with $K_0$.

After reviewing QHM and the associated Heisenberg group action, we define a smooth subalgebra $\QHMl$ of $\QHM$. Using the general proposition \ref{Prop:GenCocycleGen}, we then construct \emph{explicit} cyclic cocycles $(\varphi_i)_{i =1,2,3}$, $(\varphi_{i,3})_{i =1,2}$ and $\varphi_{1,2,3}$ out of the Lie group action (see proposition \ref{Prop:CCycle}). We rely on previous work by Abadie to obtain finitely generated projective modules $\ModNl$, $\ModNNl$ over $\QHMl$. We then construct a generating set $(U_1, U_2, U_3)$ of $K_1(\QHMl) \otimes _\Z \C$.

There are two main results. The first one consists of the theorems \ref{Thm:Bases} and \ref{Thm:Bases2}:
\begin{theoetoile}
Taking the notations of proposition \ref{Prop:CCycle} and section \ref{Sec:Unit},
\begin{itemize}
\item
the family $(\tau, \varphi_{1,3}, \varphi_{2,3})$ is a basis of $HP^0(\QHMl)$;
\item
the family $(\varphi_1, \varphi_2, \varphi_{1,2,3})$ is a basis of $HP^1(\QHMl)$;
\item
 the family $(\Ch(U_1), \Ch(U_2), \Ch(U_3))$ is a basis of $HP_1(\QHMl)$.
\end{itemize}
Moreover, if $\mu \neq 0 \neq \nu$, then the family $\left( \Ch\big(\QHMl[c][\mu,\nu] \big), \Ch \big(\ModNl[c][\mu,\nu] \big), \Ch \big( \ModNNl[c][\mu,\nu] \big) \right)$ is a basis of $HP_0(\QHMl[c][\mu,\nu])$.
\end{theoetoile}
Notice that this theorem relies on $kk$-equivalences that are constructed in \cite{ChernGrensingGabriel}. The proof also depends on the two tables \eqref{Eqn:TableauResPairs} and \eqref{Eqn:TableauResultats}:
\begin{align*}
&\begin{array}{l|c|c|c|}
\cline{2-4}
	 &\tau& \varphi_{1,3} & \varphi_{2,3} \\
	\hline 
\multicolumn{1}{|l|}{[\QHMl]} & 1	& 0 & 0 \\
   \hline
\multicolumn{1}{|l|}{[\ModNl]} &  2 \mu & - i 2 \pi & 0 \\
\hline
\multicolumn{1}{|l|}{[\ModNNl]} &  2 \nu & 0 & i 2 \pi \\
\hline
\end{array}
&
&
\begin{array}{l|c|c|c|}
\cline{2-4}
	 & \varphi_1 & \varphi_{2} & \varphi_{1,2,3}\\
	\hline 
\multicolumn{1}{|l|}{[U_1]} & -\sqrt{i 2 \pi}  & 0 & 0\\
   \hline
\multicolumn{1}{|l|}{[U_2]} & 0 &  -\sqrt{i 2 \pi}  & 0\\
\hline
\multicolumn{1}{|l|}{[U_3]} & -\sqrt{i 2 \pi} \, 2 c \nu &  \sqrt{i 2 \pi} \, 2 c \mu  & -(i 2 \pi)^{3/2}  c/3 \\
\hline
\end{array}
\end{align*}
that give the pairings of cyclic cohomology with $K$-theory in the even and odd cases. The computation of these tables is the second main result.

This paper is organized as follows. In section 2, we review the properties of QHM. Section 3 is devoted to the construction of cyclic cocycles. The pairings in the even cases are computed in section 4. Sections 5 and 6 are concerned with the calculation in the odd cases. Section 7 focuses on unfolding the consequences of the previous computations. The computation of the periodic homology and cohomology is completed in section 8. We conclude in section 9 with various remarks and perspectives.

\section{Quick Review of QHM}

We will define QHM as \emph{generalized crossed products} of $A = C(T^2)$ by certain Hilbert bimodules $\ModQHM[c][\mu,\nu]$. General references for Hilbert modules are \cite{Lance95,Black}. 

Several definitions of ``Hilbert bimodules'' exist in the literature, we use the original definition 1.8 of \cite{BMS}:
\begin{definition}[Hilbert bimodule]
A \emph{$A$-$A$-Hilbert bimodule} $E$ is a vector space such that:
\begin{itemize}
\item
$E$ is a right Hilbert $A$-module with $A$-valued scalar product $\langle \, , \, \rangle_A$;
\item
$E$ is a left Hilbert $A$-module with $A$-valued scalar product ${}_A \langle \, , \, \rangle $;
\item
the scalar products are compatible in the sense that $\xi \langle \zeta, \eta \rangle _A = {}_A \langle \xi, \zeta \rangle \eta$.
\end{itemize}
\end{definition}
Note that the above definition is \emph{not} the definition used by Pimsner in his paper \cite{PiCrG}.

The definition \ref{Def:RepCov} is closely related to theorem 3.4 of \cite{PiCrG}. However, since we consider a more special case here, we refer to \cite{BMS} definition 2.1  and \cite{FredholmExel} definition 4.5:
\begin{definition}[covariant representation]
\label{Def:RepCov}
Let $E$ be a $A$-$A$ Hilbert bimodule. A \emph{covariant representation} of $E$ on a $C^*$-algebra $B$ is a pair $(\pi,\Tt)$ which consists of:
\begin{itemize}
\item
a $*$-homomorphism of algebras $\pi \colon A \to B$;
\item 
a linear map $\Tt \colon E \to B$ satisfying
\begin{align*}
(i)& \hspace{0.25cm} \Tt(\xi)^* \Tt(\zeta) = \pi(\langle \xi , \zeta \rangle_A ) 
&
(ii)& \hspace{0.25cm} \Tt(\xi) \pi(a) = \Tt(\xi a) \\
(iii)& \hspace{0.25cm}\pi(a) \Tt(\xi) = \Tt(a \xi)
&
(iv)&\hspace{0.25cm} \Tt(\xi) \Tt(\zeta)^* = \pi({}_A \langle \xi , \zeta \rangle ).
\end{align*}
\end{itemize}
\end{definition}
We want to construct a universal $C^*$-algebra out of these representations. First notice that for any covariant representation $(\pi, \Tt)$:
\begin{align*}
\forall \xi &\in E,
&
\| \Tt(\xi) \|^2 = \| \Tt(\xi)^* \Tt(\xi) \| &= \| \pi( \langle \xi, \xi \rangle _A) \| \leqslant \| \xi \|^2,
\end{align*}
using the Hilbert module norm $\| \xi \|$. Now, given an arbitrary Hilbert bimodule $E$ over a $C^*$-algebra $A$, proposition 2.3 of \cite{AbadieEE} proves that such covariant representations exist. Together, these properties justify the following definition:
\begin{definition}[generalized crossed product]
Let $E$ be a $A$-$A$ Hilbert bimodule. The \emph{generalized crossed product} $\PiCrG $ of $A$ by $E$ is the universal $C^*$-algebra generated by the covariant representations of $E$.
\end{definition}

We now define the Hilbert bimodule $\ModQHM[c][\mu,\nu]$.
\begin{definition}[Hilbert bimodule over $A$]
\label{Def:ModQHM}
Given two real numbers $\mu, \nu$ and an integer $c>0$, we define a Hilbert bimodule $\ModQHM[c][\mu,\nu]$ over $A = C(T^2)$ as the set of continuous functions $\xi \colon \R \times S^1 \to \C$ which satisfy:
\begin{align}
\label{Eqn:Period}
\xi(x+1,y) &= e(- c (y- \nu)) \xi(x,y)
&
\xi(x,y+1)&= \xi(x,y).
\end{align}
where $e(x) = e^{2 \pi i x}$. We give it a bimodule structure by (direct) pointwise multiplication on the left and right multiplication defined by $ (\xi \cdot a)(x,y) = \xi(x,y) a(x- 2 \mu,y - 2 \nu) $. We denote by $\sigma$ the automorphism of $A$ defined by $\sigma(a)(x,y) = a(x - 2 \mu, y - 2 \nu)$. The definitions of left and right actions entail
\begin{align}
\label{Eqn:CommSigma}
&\forall \xi \in \ModQHM[c][\mu,\nu], \forall a \in A,
&
\xi a = \sigma(a) \xi.&
\end{align}
Finally, the scalar products are:
\begin{align}
\label{Eqn:ProdScalMQHM}
\langle \xi_1, \xi_2 \rangle _A(x,y) &= \overline{\xi_1(x + 2 \mu,y + 2 \nu)} \xi_2(x + 2 \mu,y + 2 \nu)
\\
{}_A \langle \xi_1, \xi_2 \rangle(x,y)  &=  \xi_1(x,y) \overline{\xi_2(x,y)},
\end{align}
\end{definition}
We can easily verify that this indeed defines a Hilbert bimodule. For convenience, we will use the shorthand notation $\ModQHM $ instead of $\ModQHM[c][\mu,\nu]$. We are now ready to define the QHM:
\begin{definition}[quantum Heisenberg manifolds]
\label{Def:QHM}
Given an integer $c>0$ and $\mu,\nu \in \R$, the \emph{quantum Heisenberg manifold} $\QHM[c][\mu,\nu]$ is the generalized crossed product $A \rtimes _{\ModQHM[c][\mu,\nu]} \Z$ of $A$ by $\ModQHM[c][\mu,\nu]$.
\end{definition}
In the following, we will write $\QHM $ instead of $\QHM[c][\mu,\nu]$ whenever possible. We will identify the elements $\xi \in \ModQHM $ with their images in $\QHM$. From the definition of $\QHM$, it appears that for any $\xi,\zeta \in \ModQHM $
\begin{align*}
{}_A \langle \xi, \zeta \rangle &= \xi \zeta^*
&
\langle \xi, \zeta \rangle_A &= \xi^* \zeta.
\end{align*}
A more explicit definition was given in \cite{RieffelDefQuant} theorem 5.5. However, we will operate on the basis of another definition (see \cite{AbadieEE}, example 3.3): $\QHM$ is the completion of the algebra
$$
 D_0 = \big\{ F \in C_c(\Z \to C_b(\R \times S^1)) \big| F(p, x + 1, y) = e(-c p(y - p \nu)) F(p,x,y) \big\} 
$$
 equipped with multiplication:
\begin{equation}
\label{Eqn:Comp}
 (F_1 \cdot F_2)(p,x,y) = \sum_{q \in \Z} F_1(q,x,y) F_2(p -q, x - q 2 \mu,y- q 2 \nu)
\end{equation}
and involution:
$$F^*(p,x,y) = \overline{F}(-p,x - 2 p \mu, y - 2 p \nu) .$$ 
Like every generalized crossed product, $\QHM$ is endowed with a gradation by $\Z$: $a \in A \subseteq \QHM$ has degree $0$ and $\xi \in \ModQHM \subseteq \QHM$ has degree $1$.
\begin{prop}
\label{Prop:Frame}
For all $\ModQHM[c][\mu,\nu]$, we can find $\xi_1, \xi_2 \in \ModQHM[c][\mu,\nu]$ such that:
\begin{align}
\label{Eqn:Frame}
\sum_{i=1}^2 \xi_i^* \xi_i = \sum_{i =1}^2 \langle \xi_i , \xi_i \rangle_A &= 1
&
\sum_{i=1}^2 \xi_i \xi_i^* = \sum_{i =1}^2 {}_A\langle \xi_i , \xi_i \rangle  &= 1.
\end{align}
\end{prop}

\begin{proof}
Let $U_1, U_2$ be the respective images of $V_1 = ]-1/3, 1/3[$ and $V_2 = ]1/6, 5/6[$ by the quotient map $\pi \colon \R \to \R / \Z \simeq S^1$. $U_1, U_2$ is an open cover of $S^1$. We call $\widetilde{\chi_1}, \widetilde{\chi_2}$ a subordinated smooth partition of the unity of $S^1$. Setting $\chi_i = \frac{\widetilde{\chi_i}}{\sqrt{\widetilde{\chi_1}^2 + \widetilde{\chi_2}^2}}$, we get $\chi_1$ and $\chi_2$ such that $ \chi_1^2 + \chi_2^2 = 1 $ and $\supp \chi_i \subseteq U_i$.

\medbreak

 To define $\xi_1$ on the cylinder $\R \times S^1$, we first set $\xi_1 (x,y) = \chi_1(x)$ on $[-1/2, 1/2] \times S^1$. This function can then be extended to the whole $\R \times S^1$ by enforcing the equations \eqref{Eqn:Period}. Notice that the extension is possible because $\xi_1$ vanishes on the boundaries of $[-1/2, 1/2] \times S^1$.

The same process can be applied to $[0,1] \times S^1$ to define $\xi_2$. In the end, we get $\xi_i$ such that 
$$ \langle \xi_i, \xi_i \rangle_A(x,y) = \chi_i(x + 2 \mu,y + 2 \nu) \overline{\chi_i(x + 2 \mu,y + 2 \nu)} = \chi_i^2(x + 2 \mu,y + 2 \nu) .$$
The first equation of \eqref{Eqn:Frame} hence comes from the property of the $\chi_i$. The same kind of computation provides the second equation.
\end{proof}

\begin{definition}[Heisenberg group $ \HH$]
The Heisenberg group $\HH$ is the subgroup of $GL_3(\R)$ of the matrices
\begin{equation}
\label{Eqn:ParamHH}
\begin{pmatrix}
1 & s & t \\
0 & 1 & r \\
0 & 0 & 1
\end{pmatrix}
, r,s,t \in \R
\end{equation}
\end{definition}
Rieffel actually uses another parametrization of $\HH$ -- see \cite{RieffelDefQuant}, p.539. He takes a nonzero $c$ and identifies $\HH$ with $\R^3$ equipped with the product
$$
	(r',s',t')(r,s,t)  = (r' + r, s' + s, t' + t + cs'r).
$$
The following definition has its origin in \cite{RieffelDefQuant}, proposition 5.6.
\begin{definition}[action of the Heisenberg group]
There is a pointwise continuous action of $\HH$ on $\QHM$ defined on $\QHM[][0]$ by:
$$ \alpha_{(r,s,t)}(F)(p,x,y) = e\big(- p(t + c s (x - r)) \big) F(p,x - r, y - s) .$$
\end{definition}
Using parameters $(r,s,t)$, the infinitesimal generators of this action are:
\begin{align*}
\partial _1 F(p,x,y) &= -\derp[F]{x}(p,x,y)
&
\partial _3(F)(p,x,y)&=- i 2 \pi p F(p,x,y)
\end{align*}
$$ 
\partial _2(F)(p,x,y) = - \derp[F]{y}(p,x,y) - i 2 \pi p c x F(p,x,y).
$$
and they fulfill the commutation relations:
\begin{align}
\label{Eqn:RelComm}
	[\partial_1, \partial_2] &= -c \partial_3 & [\partial_1, \partial_3] &= 0 & [\partial_2, \partial_3] &= 0.
\end{align}

\begin{definition}[trace on QHM]
A trace $\tau$ is given by $\tau(F) = \int_0^1 \int_{S^1} F(0,x,y) dy dx$. It is invariant under the action of $\HH$.
\end{definition}

Finally, the $K$-theory of the QHM was computed in \cite{AbadieFixedPts} (theorem 3.4):
\begin{thm}[Abadie, 1995]
\begin{align*}
K_0(\QHM) &= \Z^3 \oplus \Z / c \Z
&
K_1(\QHM)&= \Z^3.
\end{align*}
\end{thm}

\section{Cyclic Cocycles and Lie Groups Actions}

We start this section by gathering definitions and results about cyclic cohomology. We follow the exposition of \cite{NCG}. The definition of cyclic cohomology was given in \cite{NCG} III.1 p.182:
\begin{definition}[Cyclic Cohomology]
	Given an algebra $\Aa$, the \emph{cyclic cohomology} $HC^n(\Aa)$ is the cohomology of the complex $(C^n_\lambda,b)$ where $C^n_\lambda$ is the space of $(n+1)$-linear forms $\phi$ on $\Aa$ such that: 
	$$ \phi(a^0, a^1, \cdots, a^n) = (-1)^n \phi(a^n, a^0, \cdots, a^n) $$ 
	and the coboundary map $b$ is given by:
	\begin{align*}
		b \phi(a^0, \ldots, a^n, a^{n+1}) =& \sum_{j=0}^n (-1)^j \phi(a^0, \ldots, a^j a^{j+1}, \ldots, a^{n+1})\\
		&+ (-1)^{n+1} \phi(a^{n+1} a^0, \ldots, a^n).
	\end{align*}
\end{definition}
The set of all closed cochains in the above sense is denoted by $ZC^n(\Aa)$. Its elements are called \emph{cyclic cocycles}.

\begin{definition}[Cycle]
A \emph{cycle of dimension $n$} is a triple $(\Omega, d, \int)$ where $\Omega = \bigoplus_{j = 0}^n \Omega^j$ is a graded algebra over $\C$, $d$ is a graded derivation of degree $1$ such that $d^2 = 0$ and $\int \colon \Omega^n \to \C$ is a closed graded trace on $\Omega$.

A \emph{cycle over an algebra $\Aa$} is a cycle $(\Omega, d, \int)$ together with an homomorphism $\rho \colon \Aa \to \Omega^0$. The \emph{character} of such a cycle is $\phi$, the $(n+1)$-linear form on $\Aa$ given by
$$ \phi(a^0, \ldots , a^{n}) = \int \rho(a^0) d(\rho(a^1)) \cdots d(\rho(a^n)) .$$
\end{definition}

The character of a cycle essential defines the cycle. Moreover, we can restate \cite{NCG} III.1, proposition 4 :
\begin{prop}
If $\phi$ is a $(n+1)$-linear functional on $\Aa$, the following conditions are equivalent:
\begin{itemize}
\item
 $\phi$ is the character of a certain $n$-dimensional cycle $(\Omega, d, \int)$.
\item 
 $\phi$ is a cyclic cocycle.
\end{itemize}
\end{prop}

In the following definition and proposition, we adapt the construction of \cite{NCG}, III.6, example 12 c) p.254: let $A$ be a Banach algebra equipped with a pointwise continuous action $\alpha$ of the Lie group $G$. We let $ \Aa = \{ x \in A \colon  g \mapsto \alpha_g(x) \in C^\infty(G \to A) \} ,$ and $\LieAlg$ be the Lie algebra associated to $G$. 
\begin{definition}[Differential Algebra Associated to a Lie Group Action]
\label{Def:DGLieGp}
The \emph{differential algebra $\Omega_G$ associated with the action $\alpha$} is the graded differential algebra defined by the alternating $\Aa$-valued multilinear forms on $\LieAlg$ equipped with the differential $d$:
\begin{multline*}
d \omega(X_1, \ldots , X_{n+1}) = \sum_{i = 1}^{n+1} (-1)^i X_i \omega(X_1, \ldots , \check{X_i}, \ldots ,X_{n+1}) \\
+ \sum_{i < j} (-1) ^{i + j} \omega([X_i,X_j], X_1, \ldots ,\check{X_i}, \ldots,\check{X_j}, \ldots X_{n+1}).
\end{multline*}
\end{definition}

The following proposition is a restatement of the property described in \cite{NCG}, p.255. We adapt it to our simple case and state it:
\begin{prop}
\label{Prop:GenCocycleGen}
 If $\tau$ is a $G$-invariant trace over $A$ and $\xi_1 \wedge \cdots \wedge \xi_k \in \Lambda^k \LieAlg$ satisfies
\begin{align}
\label{Eqn:CondCC}
 \sum_{i < j}(-1)^{i+j} [\xi_i, \xi_j] \wedge \xi_1 \wedge \cdots \wedge \hat{\xi_i} \wedge \cdots \wedge \hat{\xi_j} \wedge \cdots \wedge \xi_n = 0,
\end{align}
then
$$ (a_0 \otimes a_1 \otimes  \cdots \otimes a_n) \mapsto \sum_{\sigma \in \Sigma_n} \epsilon(\sigma) \tau \big(a_0 {\xi_{\sigma(1)}}(a_1) \cdots{\xi_{\sigma(n)}}(a_n) \big) $$
is a cyclic cocycle on $\Aa$.
\end{prop}

In order to use definition \ref{Def:DGLieGp} and proposition \ref{Prop:GenCocycleGen}, we must take:
\begin{definition}[smooth QHM]
\label{Def:SmoothQHM}
The \emph{smooth QHM} $\QHMl$ is defined as $ \QHMl = \{ F \in \QHM \colon  g \mapsto \alpha_g(F) \in C^\infty(\HH \to \QHM) \} $.
\end{definition}

As $\alpha$ is pointwise continuous, proposition 3.45 p.138 of \cite{EltNCG} proves that $\Aa$ is dense in $A$, and also that it is stable under holomorphic calculus -- such algebras are called \emph{pre-$C^*$-algebras} in \cite{EltNCG}, see definition 3.26 p.134. It is a well known fact that for such subalgebras the injection $\Aa \inj A$ induces an isomorphism in $K$-theory (see for instance \cite{EltNCG} theorem 3.44 p.138). Hence,
\begin{lem}
$\QHMl$ is a dense subalgebra that is stable under holomorphic functional calculus and therefore
\begin{align}
\label{Eqn:KThQHM}
K_0(\QHMl)&= K_0(\QHM) = \Z^3 \oplus \Z/c \Z
&
K_1(\QHMl) &= K_1(\QHM) = \Z^3.
\end{align}
\end{lem}

\begin{prop}
\label{Prop:CCycle}
The 7 following multilinear forms are cyclic cocycles on $\QHMl$:
\begin{itemize}
\item
Degree 0: trace $\tau$.
\item
Degree 1: $\varphi_i$ for $i =1, 2, 3$ where $ \varphi_i (a_0, a_1) = \tau(a_0 \partial _i(a_1)) $.
\item
Degree 2: $\varphi_{1,3}$ and $\varphi_{2,3}$ where
$$ \varphi_{i,3}(a_0, a_1,a_2) = \tau\Big(a_0 \big(\partial_i(a_1) \partial_3(a_2) - \partial_3(a_1) \partial_i(a_2)\big)\Big).$$
\item
Degree 3: $\varphi_{1,2,3}$ given by:
$$ \varphi_{1,2,3}(a_0, a_1, a_2, a_3) = \sum_{\sigma \in \Sigma_3} \varepsilon(\sigma) \tau \Big ( a_0 \partial_{\sigma(1)} a_1 \partial_{\sigma(2)} a_2 \partial_{\sigma(3)} a_3 \Big ) .$$
\end{itemize}
\end{prop}

\begin{proof}
This is a straightforward application of proposition \ref{Prop:GenCocycleGen}.
\begin{itemize}
\item
Degree 0: there is nothing to check since $\tau$ is a trace.
\item
Degree 1: any derivation generates a cyclic cocycle. 
\item
Degree 2: the commutation relations \eqref{Eqn:RelComm} show that $\xi_1 \wedge \xi_3$ and $\xi_2 \wedge \xi_3$ satisfy the condition \eqref{Eqn:CondCC}.
\item
Degree 3: start with $\xi_1 \wedge \xi_2 \wedge \xi_3$. The condition can therefore be written:
$$ (-1)^3 [\xi_1, \xi_2] \wedge \xi_3 + (-1)^4 [\xi_1, \xi_3] \wedge \xi_2 + (-1)^5 [\xi_2, \xi_3] \wedge \xi_1 = 0. $$
The commutation relations \eqref{Eqn:RelComm} ensure that this expression vanishes.
\end{itemize}
\end{proof}

\begin{remarque}
One could expect a third $2$-cyclic cocycle $\varphi_{1,2}$ given by $ \varphi_{1,2}(a_0, a_1,a_2) = \tau\Big(a_0 \big(\partial_1(a_1) \partial_2(a_2) - \partial_2(a_1) \partial_1(a_2)\big)\Big)$. $\varphi_{1,2}$ is indeed a \emph{Hochschild} cocycle, as an easy computation shows. However, proposition \ref{Prop:DualCycles} implies that $\varphi_{1,2}$ is \emph{not} a \emph{cyclic} cocycle.
\end{remarque}

\begin{proof}
 Indeed, if $\varphi_{1,2}$ were a cyclic cocycle, then $\varphi_{1,2}(1,a_1,a_2) = 0$. Yet,
\begin{multline*}
\varphi_{1,2}(1, a_1,a_2) = \tau\big( \partial_1(a_1) \partial_2(a_2) - \partial_2(a_1) \partial_1(a_2)\big) = \\
=\tau \big( \partial _1(a_1 \partial _2(a_2)) - \partial _2(a_1 \partial _2(a_2)) \big) - \tau \big( a_1 \partial _1 \partial _2(a_2) - a_1 \partial _2 \partial _1(a_2)\big) =\\
= c \tau\big(a_1 \partial _3(a_2) \big) = c \varphi_3(a_1,a_2),
\end{multline*}
using \eqref{Eqn:RelComm}. But the proposition \ref{Prop:DualCycles} proves that $\varphi_3$ pairs nontrivially with a Hochschild cocycle and therefore is nonzero.
\end{proof}

\section{Modules and even Pairings}

Bear in mind (see \cite{IHES} part II, theorem 9) that if $ a_i \otimes b_i \in A \otimes M_n(\C)$ and $\phi \in HC^*$, the cup product $\phi \# \tr$ is defined by:
$$ (\phi \# \tr)(a^0 \otimes b^0, \ldots, a^n \otimes b^n) = \phi(a^0, \ldots, a^n) \tr(b^0 \cdots b^n). $$

We follow the definition and normalizations of \cite{NCG} (III.3 proposition 2): 
\begin{definition}[even Chern-Connes Pairings]
\label{Def:DefAppariement}
The following formula defines a bilinear pairing between $K_0(\Aa)$ and $HC^{2 m}(\Aa)$:
\begin{equation*}
\langle [e],[\phi]\rangle  = \frac{1}{m !}(\phi \# \tr)(e, \ldots,e) ,
\end{equation*}
where $[e] \in K_0(\Aa)$ and $ \phi \in ZC^{2 m}(\Aa)$.
\end{definition}

Furthermore, there is a \emph{periodicity map} $S \colon HC^n(\Aa) \to HC^{n+2}(\Aa)$ which enables us to define the groups $HP^*(\Aa)$ -- see \cite{EltNCG}, 10.1, definition 10.5 p.445:
\begin{definition}[periodic cyclic cohomology]
 The \emph{periodic cyclic cohomology} is composed of two groups obtained as inductive limits:
\begin{align*}
HP^0(\Aa) &= \lim_{\to} HC^{2 k}(\Aa)
&
HP^1(\Aa) &= \lim_{\to} HC^{2 k+1}(\Aa).
\end{align*}
\end{definition}
The above pairings are in fact defined on $HP^0(\Aa)$ because they satisfy $\langle [e], [S \phi] \rangle = \langle [e], [\phi] \rangle $.

Before stating our theorem on even pairings of QHM, we need to introduce elements of $K_0(\QHMl)$, \textsl{i.e.} finitely generated projective modules over $\QHMl$. For brevity, we will call these \emph{finite projective modules}.

These modules over $\QHM[c][\mu,\nu]$ were studied by Abadie in \cite{AbadieVectBundle} and \cite{AbadieFixedPts}. We present the results of \cite{AbadieVectBundle} pp. 2--3:
\begin{thm}[Abadie, 1992]
\label{Thm:Abadie}
For all $c \in \N^*$, $\mu, \nu \in \R$ such that $\mu^2 + \nu^2 \neq 0$, there is a finite projective right module $\ModN[c][\mu,\nu]$ over $\QHM[c][\mu,\nu]$. $\ModN[c][\mu,\nu]$ is obtained by completing $C_c(\R \times S^1)$ with respect to the $\QHM[c][\mu, \nu]$-valued scalar product:
$$ \langle  f,g \rangle _{\QHM[c][\mu,\nu]}(p,x,y) = \sum_{n \in \Z} e(c n p(y -p \nu)) \overline{f(x+n,y)} g(x - 2 p \mu + n, y - 2 p \nu) .$$
The right action of $\QHM[c][\mu,\nu]$ is given by
$$ (f\cdot F)(x,y) = \sum_q f(x-2 q \mu, y - 2 q \nu) F(-q,x-2 q \mu, y - 2 q \nu) $$
Moreover $\tr \big(\id_{\ModN[c][\mu,\nu]}\big) = 2 \mu$.
\end{thm}

We will write $\ModN$ instead of $\ModN[c][\mu,\nu]$ whenever possible.

\begin{remarque}
This module $\ModN$ is in fact the dual to the module $X$ given in \cite{AbadieVectBundle}, p.2.
\end{remarque}

We will see later (definitions \ref{Def:ModNN}, \ref{Def:ModNl} and \ref{Def:ModNNl}) that we can define 
\begin{itemize}
\item
another module $\ModNN$ based on $\ModN$ and an isomorphism between two QHM,
\item
 and ``smooth versions'' $\ModNl$ and $\ModNNl$ of $\ModN$ and $\ModNN$.
\end{itemize}
We are now ready to state:
\begin{thm}
If $\mu \neq 0 \neq \nu$, we can define finite projective modules $\ModN$ and $\ModNN$ over $\QHM $, and the values of the pairings are given by the table:
\begin{equation}
\label{Eqn:TableauResPairs}
\begin{array}{l|c|c|c|}
\cline{2-4}
	 &\tau& \varphi_{1,3} & \varphi_{2,3} \\
	\hline 
\multicolumn{1}{|l|}{[\QHMl]} & 1	& 0 & 0 \\
   \hline
\multicolumn{1}{|l|}{[\ModNl]} &  2 \mu & - i 2 \pi & 0 \\
\hline
\multicolumn{1}{|l|}{[\ModNNl]} &  2 \nu & 0 & i 2 \pi \\
\hline
\end{array}
\end{equation}
\end{thm}

Notice that: 
\begin{itemize}
\item
the first column of this table was computed by B. Abadie in \cite{AbadieTraceRange};
\item 
 the first line is easy to compute directly from the definition \ref{Def:DefAppariement} with $e = 1$;
\item 
if $\mu = 0$ or $\nu = 0$, then the finite projective modules $\ModNl$ and $\ModNNl$ are not both defined. In this case, we have to use the isomorphism between $\QHM[c][\mu+1,\nu+1]$ and $\QHM[c][\mu,\nu]$ that was proved by Abadie in  \cite{AbadieVectBundle}. This phenomenon of ``disappearing module'' may be surprising, but it also happens with ``Schwartz modules'' in the case of noncommutative tori -- see \cite{IHES} part II, definition above lemma 54.
\end{itemize}

	\subsection{connections and Pairings for $\ModNl$}

The definition of a (noncommutative) connection is given in \cite{NCG}, III.3. definition 5:
\begin{definition}[connection]
Let $\Aa \xrightarrow{\rho} \Omega$ be a cycle over $\Aa$, and $\Modl $ a finite projective right module over $\Aa$, a \emph{connection} $\nabla$ on $\Modl $ is a linear map $\nabla \colon \Modl  \to \Modl  \otimes_\Aa \Omega^1$ such that:
\begin{align*}
&\forall \xi \in \Modl , \forall a \in \Aa,
&
\nabla(\xi a) &= (\nabla \xi) a + \xi \otimes d \rho(a)
\end{align*}
\end{definition}
Notice that in the above definition, we do not actually need the trace $\int$ of the cycle. We could equally consider a ``graded differential algebra over $\Aa$''.

The following proposition enables us to compute the even pairings using the connections -- see \cite{NCG}, III.3 proposition 8:
\begin{prop}
\label{Prop:Calcconnection}
Let $\Modl$ be a finite projective right module over $\Aa$. Assume that we have a cycle $(\Omega, d, \int)$ over $\Aa$.

\begin{enumerate}
\item
$\widetilde{\Modl} = \Modl \otimes_\Aa \Omega$ is a finite projective $\Omega$-module.
\item 
All connection $\nabla$ can be uniquely extended to a linear application from $\widetilde{\Modl}$ into itself that satisfies~:
\begin{align*}
&\forall \xi \in \Modl, \omega \in \Omega, 
&
\nabla(\xi \otimes \omega) &= (\nabla \xi) \omega + \xi \otimes d \omega .
\end{align*}
\item We have
$$\langle [\Modl],[\phi] \rangle  = \frac{1}{m !}\int \theta^m $$
where
\begin{itemize}
\item $n$ is the (even) dimension of the cycle $\Omega$ and $n = 2 m$;
\item $[\Modl] \in K_0(\Aa)$ is the class of $\Modl$;
\item $\phi$ is the character of $\Omega$;
\item $\int$ is the trace of the cycle;
\item $\theta$ is the endomorphism of $\widetilde{\Modl }$ defined by $\theta = \nabla^2$.
\end{itemize}
\end{enumerate}
\end{prop}

\begin{definition}[Covariant Action]
Let $E$ be a right module over a $C^*$-algebra $A$. A \emph{covariant action} of a Lie group $G$ on $E$ is a pair $(\alpha, \beta)$ where $\alpha$ and $\beta$ are actions of $G$ on $A$ and $E$ respectively, which satisfy
\begin{align*}
&\forall \xi \in E, \forall a \in A,
&
\beta_g(\xi a) &= \beta_g(\xi) \alpha_g(a).
\end{align*}
\end{definition}

If we have a covariant action on a module,
\begin{itemize}
\item
 we can define a ``smooth version'' of the module on the ``smooth version'' of the algebra;
\item 
 the smooth module is equipped with a connection.
\end{itemize}
To state this formally:
\begin{prop}
\label{Prop:ActCovconnection}
Let $(\alpha,\beta)$ be a covariant action on a right $A$-module $E$. The subvector space $\Modl \subseteq E$ defined by
$$ \Modl = \big\{ \xi \in E \colon g \mapsto \beta_g(\xi) \text{ is in }C^\infty(G \to E) \big\} $$
has a natural left action by $\Aa \subseteq A$ and a connection $\nabla \colon \Modl \to \Modl \otimes \AlgLie^*$ over the differential algebra $\Omega_G$ given by
\begin{align*}
&\forall X \in \AlgLie^*,
&
\nabla_X \xi = \lim_{t \to 0} \frac{\beta_{e^{t X}}(\xi) - \xi}{t}.&
\end{align*}
\end{prop}

\begin{proof}
Let us first check that $\Modl $ is a module over $\Aa$: if $\xi \in \Modl $, then $ \xi a \in \Modl $ because $g \mapsto \beta_g(\xi a) = \beta_g(\xi) \alpha_g(a)$.

Now we must prove that $\nabla$ is a connection: 
\begin{multline*}
 \nabla_X (\xi a) = \lim_{t \to 0} \frac{\beta_{e^{t X}}(\xi a) - \xi a}{t} = \lim_{t \to 0} \frac{\beta_{e^{t X}}(\xi) \alpha_{e^{t X}}( a) - \xi a}{t} = \\
= \lim_{t \to 0} \frac{1}{t} \left( \beta_{e^{t X}}(\xi) \alpha_{e^{t X}}( a) - \xi \alpha_{e^{t X}}( a) + \xi \alpha_{e^{t X}}( a)- \xi a \right) = (\nabla_X \xi ) a + \xi \partial _X (a).
\end{multline*}
\end{proof}

\begin{prop}
\label{Prop:ActCovModN}
If $\mu \neq 0$, there is a covariant action $\beta$ of $\HH$ on $\ModN$ given by
$$ \beta_{(r,s,t)}(f)(x,y) = e^{i x \frac{\pi}{\mu} \Big(t + s c \big( x/2 -r  - \mu q\big) \Big)} f(x- r,y-s) .$$
\end{prop}

\begin{proof}
On the one hand,
\begin{multline*}
\beta_{(r,s,t)}(f\cdot F)(x,y) =\\
 e^{i x \frac{\pi}{\mu} \Big(t + s c \big( x/2 -r  - \mu q\big) \Big)} \sum_q  f(x -r -2 q \mu, y - s - 2 q \nu) F(x-r -2 q \mu, y -s - 2 q \nu, -q) .
\end{multline*}
On the other hand,
\begin{multline*}
\left ( \beta_{(r,s,t)}(f) \cdot \alpha_{(r,s,t)}(F) \right )(x,y) = \\
\sum_{q} e^{i (x- 2 q \mu) \frac{\pi}{\mu} \Big(t + s c \big( (x - 2 q \mu)/2 - r  - \mu q\big) \Big)} f(x- r - 2 q \mu,y-s - 2 q \nu)\\ 
\times e\Big(q \big(t + c s (x - r - 2 q \nu)\big)\Big) F(x-r -2 q \mu, y - s - 2 q \nu, -q).
\end{multline*}
We can restrict this to study the phase factor
$$e^{i (x- 2 q \mu) \frac{\pi}{\mu} \Big(t + s c \big( (x - 2 q \mu)/2 - r  - \mu q\big) \Big)} e^{ i  2 \pi q\big(t + c s (x - r - 2 q \nu)\big)}$$
while ignoring the $e^{i \cdot }$:
\begin{multline*}
(x- 2 q \mu) \frac{\pi}{\mu} \Big(t + s c \big( (x - 2 q \mu)/2 -r  - \mu q\big) \Big) + 2 \pi q \big(t + c s (x - r - 2 q \nu)\big) = \\
= t \left ( \frac{\pi}{\mu}(x - 2 q \mu) + 2 \pi q \right ) + s c \left ( (x - 2 q \mu) \frac{\pi}{\mu} \big( (x - 2 q \mu)/2 -r  - \mu q\big)  + 2 \pi q (x - r  - 2 q \mu)  \right ) = \\
= t  \frac{\pi}{\mu}x  - s c r \left ( \frac{\pi}{\mu} (x - 2 q \mu) + 2 \pi q \right ) + s c  (x - 2 q \mu) \frac{\pi}{\mu}((x - 2 q \mu)/2  - \mu q + 2 q \mu) = \\
= t  \frac{\pi}{\mu}x  - s c r \frac{\pi}{\mu} x + s c (x - 2 q \mu) \frac{\pi}{\mu} x/2 = x \frac{\pi}{\mu} \left (t + s c (x/2 - r - q \mu) \right ).
\end{multline*}
Integrating this property into the expression of $ \beta_{(r,s,t)}(f) \cdot \alpha_{(r,s,t)}(F)$:
\begin{multline*}
\left ( \beta_{(r,s,t)}(f) \cdot \alpha_{(r,s,t)}(F) \right )(x,y) = \\
= \sum_q e^{i x \frac{\pi}{\mu} \left (t + cs (x/2 - r - q \mu) \right )} f(x-r -2 q \mu,y-s- 2 q \nu) F(x - r - 2 q \mu,y-s - 2 q \nu,-q) =\\
= \beta_{(r,s,t)}(f\cdot F)(x,y).
\end{multline*}
\end{proof}

\begin{definition}
\label{Def:ModNl}
We denote by $\ModNl$ the $\QHMl$-module of the elements of $\ModN$ that are regular under the action $\beta$ of $\HH$. $\ModNl$ is a finite projective module over $\QHMl$.
\end{definition}

Using the covariant action on $\ModN $, it is easy to construct connections over $\ModNl$:
\begin{prop}
The connections over $\ModNl$ associated to $\varphi_{1,3}$ and $\varphi_{2,3}$ respectively, are:
$$ (\nabla f)(x,y) = -\derp[f]{x}(x,y) dx + \frac{i \pi}{ \mu} x f(x,y) dp $$
and
$$ (\nabla f)(x,y) = \left(-\derp[f]{y} + i \pi c x \left (\frac{x}{2 \mu}- q \right )f \right)(x,y) dy + \frac{i \pi}{ \mu} x f(x,y) dp .$$
\end{prop}

\begin{proof}
This is an obvious application of propositions \ref{Prop:ActCovconnection} and \ref{Prop:ActCovModN}.
\end{proof}

\begin{prop}
\label{Prop:AppModN}
If $\mu \neq 0$, the values of the pairings over $\ModNl$ are:
\begin{align*}
\langle [\ModNl], [\varphi_{1,3}] \rangle &=- i 2 \pi
&
\langle [\ModNl], [\varphi_{2,3}] \rangle &= 0.
\end{align*}
\end{prop}

\begin{proof}
First consider the case of $\varphi_{1,3}$. Computing $\nabla^2 f$ yields: 
\begin{align*}
	\nabla \big( \nabla f \big) =& \nabla \left( -\derp[f]{x} dx + \frac{i \pi}{ \mu} x f dp \right) \\
&= \left(\frac{\partial^2 f}{\partial x^2} dx - \frac{i \pi}{ \mu} x \derp[f]{x} dp\right) dx + \frac{i \pi}{ \mu} \left(  - \left(f + x \derp[f]{x} \right)dx + \frac{i \pi}{ \mu}  x^2 f dp \right) dp \\
&= - i \frac{\pi}{\mu} f \otimes dx \wedge dp
\end{align*}
thus $\nabla^2 = - i \frac{\pi}{\mu} \id_{\ModNl} \otimes dx \wedge dp $. The trace of $\id_{\ModN}$ was computed by Abadie (theorem \ref{Thm:Abadie}). The pairing is given by the table \eqref{Eqn:TableauResPairs}: $ \tau \left( p \right) = 2 \mu $. Hence
$$ \langle [\ModNl], \varphi_{1,3}\rangle  = - i 2 \pi .$$
In the case of $\varphi_{2,3}$, the two terms in the connection commute, so $\nabla^2 = 0$ and
$$ \langle [\ModNl], \varphi_{2,3} \rangle  = 0 .$$
\end{proof}

	\subsection{connections and Pairings for $\ModNNl$}

We define a second module $\ModNN$ over $\QHM $ through an isomorphism between two $\QHM$: 
\begin{prop}
\label{Prop:IsomQHMmunu}
There is an isomorphism $\Phi \colon \QHM[c][\mu,\nu] \to \QHM[c][\nu,\mu]$, induced by the Hilbert bimodule representation:
\begin{align*}
\pi(a)(x,y) &= a(-y,-x) 
&
\Tt(\xi)(x,y) &= e^{i 2 \pi c (y + \mu)(x+\nu)} \xi(-y,-x)
\end{align*}
Moreover, this isomorphism intertwines the actions of $\HH$:
\begin{equation}
\label{Eqn:EntrelacementHH}
\Phi(\alpha_{r,s,t}(F)) = \alpha'_{-s, -r, -t - c ( \mu s + r \nu - r s)}(\Phi(F))
\end{equation}
where $\alpha$ and $\alpha'$ are the actions of $\HH$ over $\QHM[c][\mu,\nu]$ and $\QHM[c][\nu,\mu]$, respectively.
\end{prop}

Before starting the proof let us indicate two points.
\begin{itemize}
\item
  The intertwining of the actions proves that $\QHMl[c][\mu,\nu]$ and $\QHMl[c][\nu,\mu]$ are sent onto one another.
\item 
 The existence of such isomorphism was proved in \cite{AbadieExel} theorem 2.2. However, we here give an \emph{explicit} isomorphism and specify the intertwining relation. 
\end{itemize}

\begin{proof}
Following the definition \ref{Def:QHM}, the algebra $\QHM[c][\nu,\mu,0]$ of continous functions $\Z \times \R \times \R \to \C$ which satisfy
\begin{align*}
F(p,x + 1,y) &= e\big(-c p (y - p \mu)\big) F(p,x,y) & F(p,x,y+1) &= F(p,x,y).
\end{align*}
is dense in $\QHM[c][\nu,\mu]$. The definition \ref{Def:QHM} implies that it suffices to check that there is a representation of the Hilbert bimodule $\ModQHM[c][\mu,\nu]$ in $\QHM[c][\nu,\mu]$ to obtain the morphism of algebra $\Phi$.

Let us check that $\Tt(\xi)$ has the degree $-1$ in $\QHM[c][\nu,\mu,0]$:
$$
\Tt(\xi)(x+1,y) =e^{i 2 \pi c (y + \mu)(x+1+\nu)} \xi(-y,-x-1) = e^{i 2 \pi c (y + \mu)} \Tt(\xi)(x,y)
$$
and
\begin{multline*}
\Tt(\xi)(x,y+1) = e^{i 2 \pi c (y+1 + \mu)(x+\nu)} \xi(-y-1,-x) =\\
= e^{i 2 \pi c(x+ \nu)} e^{i 2 \pi c (y + \mu)(x+\nu)} e^{i 2 \pi c (-x - \nu)} \xi(-y,-x) = \Tt(\xi)(x,y).
\end{multline*}
To prove that $\pi$ and $\Tt$ induce a bimodule representation, it suffices to prove the points (i) and (iv) of definition \ref{Def:RepCov}. The others are consequences of these two (see \cite{Katsura04} definition 2.1).

Regarding point (i): 
\begin{multline*}
\Tt(\xi)^* \Tt(\zeta)(0,x,y) =\\
=  e^{-i 2 \pi c (y - \mu)(x-\nu)} \overline{\xi(-y + 2 \mu,-x + 2 \nu)} e^{i 2 \pi c (y - \mu)(x-\nu)} \zeta(-y + 2 \mu,-x  + 2 \nu) =\\
= \overline{\xi(-y + 2 \mu,-x + 2 \nu)}\zeta(-y + 2 \mu, -x +2 \nu) =\pi( \langle \xi, \zeta \rangle_A )(x,y).
\end{multline*}
As for point (iv):
\begin{multline*}
\Tt(\zeta)\Tt(\xi)^*(0,x,y) = e^{i 2 \pi c (y + \mu)(x+\nu)} \zeta(-y, -x) e^{-i 2 \pi c (y + \mu)(x+\nu)} \overline{\xi(-y,-x)} =\\
= \zeta(-y,-x) \overline{\xi(-y,-x)} = \pi \left ({}_A \langle \zeta, \xi \rangle \right )(x,y).
\end{multline*}
Hence, there is a homomorphism from $\QHM[c][\mu,\nu]$ into $\QHM[c][\nu,\mu]$.

\medbreak

$\pi(A)$ and $\Tt(E)$ generate the algebra, hence it is sufficient to check the intertwining \eqref{Eqn:EntrelacementHH} on these two sets. Denoting by $\alpha'$ the action of $\HH$ on $\QHM[c][\nu,\mu]$, we get:
$$ \pi(\alpha_{r,s,t}(a))(x,y) = \pi( a(x - r, y - s)) = a(-y - r, - x -s) = \alpha'_{-s, -r, -t - c ( \mu s + r \nu - r s)}(\pi(a))(x,y) $$
as well as:

\begin{align*}
\Tt(\alpha_{r,s,t} \xi)&(x,y) = \Tt \Big(  e\big(-\big(t + c s (x - r)\big)\big) \xi(x-r,y-s) \Big)  \\
&=e^{i 2 \pi c (y + \mu)(x+\nu)}e\Big(-\big(t + c s (-y - r)\big)\Big)  \xi(-y-r,-x-s) \\
&=e\big(-t + c (y + \mu+r)(x+\nu+s) - c r (x + \nu+s) - c (y +\mu) s + c s (y+r) \big) \\
&\hspace{5cm} \times \xi( - y - r, - x - s) \\
&=e\big(-t + c (y + \mu+r)(x+\nu+s) - c r (x + \nu) - c \mu s \big) \xi( - y - r, - x - s) \\
&=e\big(-t - c ( \mu s + r \nu - r s) - c s (x + r)) \big) e^{i 2 \pi c (y + \mu + r) (x + \nu + s)}\xi( - y - r, - x - s) \\
&=\big(\alpha'_{-s, -r, -t - c ( \mu s + r \nu - r s)} \Tt(\xi) \big)(x,y).
\end{align*}
\end{proof}

\begin{definition}[module ${\ModNN[c][\mu,\nu]}$]
\label{Def:ModNN}
We define the right module $\ModNN[c][\mu,\nu]$ over $\QHM[c][\mu,\nu]$ as the module induced from $\ModN[c][\nu,\mu]$ through $\Phi$. This module exists as soon as $\nu\neq 0$. Explicitely, the action of $\QHM[c][\mu,\nu]$ is given by $ f \cdot F = f \, \Phi(F) $.
\end{definition}

To unclutter notations, we will write $\ModNN$ instead of $\ModNN[c][\mu,\nu]$ whenever possible.

\begin{remarque}
\label{Rem:ModInduit}
If $P'$ is a projector of $M_n(\QHM[c][\nu,\mu])$ associated to $\ModN[c][\nu,\mu]$, then $P = \Phi^{-1}(P')$ is a projector of $M_n(\QHM[c][\mu,\nu])$ associated to $\ModNN[c][\mu,\nu]$.
\end{remarque}

\begin{lem}
\label{Lem:ImCCcocycle}
The cyclic cocycles $\varphi_{i,3}$ induced on $\QHMl[c][\mu,\nu]$ from the cyclic cocycles $\varphi'_{i,3}$ over $\QHMl[c][\nu,\mu]$ are given by:
\begin{align*}
\Phi^* \varphi'_{1,3} &= -\varphi_{2,3} 
&
\Phi^* \varphi'_{2,3} &= -\varphi_{1,3} 
\end{align*}
\end{lem}

\begin{proof}
By taking the derivative of \eqref{Eqn:EntrelacementHH}, one obtains:
\begin{align*}
\partial' _1  \circ \Phi &= \Phi \circ \big( c \mu \partial_3 - \partial _2 \big)
& 
\partial' _2  \circ \Phi &= \Phi \circ \big( c \nu \partial_3 - \partial _1 \big)
&
 \partial' _3 \circ \Phi &=-\Phi \circ \partial_3 
\end{align*}
linking the derivations $\partial' _i$ of $\QHMl[c][\nu,\mu]$ and $\partial _i$ of $\QHMl[c][\mu,\nu]$. The traces $\tau^c_{\mu,\nu}$ and $\tau^c_{\nu,\mu}$ are invariant in the sense that $ \tau^c_{\nu,\mu} \left ( \Phi(F) \right ) = \tau^c_{\mu,\nu}(F) $. Consequently, the induced cocycles are: 
\begin{multline*}
 \Phi^* \varphi'_{i,3}(a_0,a_1,a_2) = \tau \Big( \Phi(a_0) \big[\partial '_i(\Phi(a_1)) \partial '_3(\Phi(a_2)) - \partial '_3(\Phi(a_1)) \partial '_i(\Phi(a_2) \big] \Big) = \\
= \tau \Big( \Phi \big[ a_0 (k_i \partial _3 - \partial _j)(a_1) \partial _3(a_2) - \partial _3(a_1)  (k_i \partial _3 - \partial _j)(a_2) \big] \Big) =\\
=- \varphi_{j,3}(a_0,a_1, a_2),
\end{multline*}
where $(i,j)$ is a permutation of $(1,2)$ and $k_1 = c \mu$, $k_2 = c \nu$.
\end{proof}

It is easy but fastidious to check that $(r,s,t) \mapsto \alpha'_{-s,-r, -t - c ( \mu s + r \nu - r s)}$ is indeed a representation of $\HH$ on $\QHM[c][\nu,\mu]$. It then appears from propositions \ref{Prop:ActCovModN} , \ref{Prop:IsomQHMmunu} and definition \ref{Def:ModNN} that $\ModNN$ can be equipped with a covariant action $(\alpha, \beta^\dag)$ defined by:
$$ \beta^\dag_{(r,s,t)}(f)= \beta'_{-s,-r, -t - c ( \mu s + r \nu - r s)}(f) ,$$
where $\beta'$ is the action $\HH \action \ModN[c][\nu,\mu]$ given in proposition \ref{Prop:ActCovModN}. This justifies 
\begin{definition}
\label{Def:ModNNl}
We denote by $\ModNNl$ the $\QHMl$-module of the elements of $\ModNN$ that are regular under the action $\beta^\dag$ of $\HH$. $\ModNNl$ is a finite projective module over $\QHMl$.
\end{definition}

\begin{prop}
If $\nu \neq 0 \neq \mu$, the pairings of $\ModNNl$ are given by 
\begin{align*}
	 \langle [\ModNNl], \varphi_{1,3} \rangle  &= 0 & 
	 \langle [\ModNNl], \varphi_{2,3} \rangle  &= i 2 \pi.
\end{align*}
\end{prop}

\begin{proof}
Considering remark \ref{Rem:ModInduit}, the definition of $\Phi^* \varphi_{i,3}$ and the lemma \ref{Lem:ImCCcocycle}, we clearly have
$$ \langle [\ModNNl[c][\mu,\nu]], \varphi_{j,3} \rangle =- \langle [\ModNNl[c][\mu,\nu]], \Phi^* \varphi_{i,3}' \rangle = -\langle [\ModNl[c][\nu,\mu]], \varphi_{i,3} \rangle ,$$
where $(i,j)$ is a permutation of $(1,2)$. Since proposition \ref{Prop:AppModN} gives us the values of the pairings $\langle [\ModNl[c][\nu,\mu]],\varphi_{1,3}' \rangle$ and $\langle [\ModNl[c][\nu,\mu]],\varphi_{2,3}' \rangle$, we get
\begin{align*}
	 \langle [\ModNNl], \varphi_{1,3} \rangle  &= 0 & 
	 \langle [\ModNNl], \varphi_{2,3} \rangle  &=  i 2 \pi.
\end{align*}
\end{proof}

\section{Unitaries and First Odd Pairings}
\label{Sec:Unit}

The inclusion $A \inj \QHM[c][\mu,\nu]$ induces 2 elements of $K_1(\QHM[c][\mu,\nu])$. Using $e(x) = e^{i 2 \pi x}$, we can express them as
\begin{align*}
U_1(p,x,y) &= \delta_{0,p} e(x) &
U_2(p,x,y) &= \delta_{0,p} e(y) &
\end{align*}
We want to construct a third unitary. As a first step, we extend \eqref{Eqn:CommSigma} to describe how elements of $\ModQHM[c][\mu,\nu]$ and their adjoints commute:
\begin{lem}
\label{Lem:RelCommSigma}
For all $\xi, \zeta \in \ModQHM[c][\mu,\nu] \subseteq \QHM[c][\mu,\nu]$, $\sigma(\xi^* \zeta) =  \zeta \xi^* $.
\end{lem}

\begin{proof}
We know from \eqref{Eqn:CommSigma} that $a \xi = \xi \sigma(a) $ for all $a \in A$. Using the elements $\xi_1, \xi_2$ of proposition \ref{Prop:Frame}, we can write:
$$
\sigma(\xi^* \zeta) (\xi_1\xi_1^* + \xi_2 \xi_2^*) = \xi_1 \xi^* \zeta \xi_1^* +\xi_2 \xi^* \zeta \xi_2^* = \zeta \xi_1^* \xi_1 \xi^* + \zeta \xi_2^*\xi_2 \xi^*  = \zeta \xi^*,
$$
using the commutativity of $A$.
\end{proof}

\begin{remarque}
The automorphism $\sigma$ of $A$ can be extended to an automorphism of $\QHM$, also noted $\sigma$. We need only set $\sigma = \alpha_{(2 \mu,2 \nu,0)}$.
\end{remarque}

\begin{prop}
\label{Prop:RelUnit}
Let
\begin{align*}
M_+ &= 
\begin{pmatrix}
 \xi_1 & 0 \\
- \xi_2 & 0
\end{pmatrix}
&
M_- &= \begin{pmatrix}
0 & \sigma(\xi_2)^* \\
0 & \sigma(\xi_1)^* 
\end{pmatrix}.
\end{align*}
We get the relations:
\begin{align*}
M_+ M_+^* &= P_+ =\begin{pmatrix}
\xi_1 \xi_1^* & - \xi_1 \xi_2^* \\
- \xi_2 \xi_1^* & \xi_2 \xi_2^*
\end{pmatrix} &
M_+^* M_+ &= Q_+ = \begin{pmatrix}
1 & 0 \\
0 & 0
\end{pmatrix}
\\
M_- M_-^* &= P_- = \begin{pmatrix}
\sigma(\xi_2^* \xi_2) & \sigma(\xi_2^* \xi_1) \\
\sigma(\xi_1^* \xi_2) & \sigma(\xi_1^* \xi_1) \\
\end{pmatrix} & M_-^* M_- &= Q_- = \begin{pmatrix}
0 & 0 \\ 
0 & 1
\end{pmatrix}
\end{align*}
and
\begin{align*}
P_\pm = P_\pm^2 &= P_\pm^* & P_+ + P_- &= I_2
&
Q_\pm = Q_\pm^2 &= Q_\pm^* & Q_+ + Q_- &= I_2\\
P_\pm M_\pm &= M_\pm & M_\pm Q_\pm &= M_\pm 
& 
M_\pm M_\mp^* &= 0 &  M_\pm^* M_\mp &= 0.
\end{align*}
Thus $U_3 = M_+ + M_-$ is a unitary.
\end{prop}

We will compute the index of $U_3$ in proposition \ref{Prop:Indice}.

\begin{proof}
The first four relations are proved by direct calculation. Those involving $Q_\pm$ are obvious, because of \eqref{Eqn:Frame}. The $P_\pm$ are clearly self-adjoint. A direct computation using $\xi_1\xi_1^*  + \xi_2\xi_2^*  = \sum_{i =1}^2 {}_A \langle \xi_i, \xi_i \rangle = 1$ shows that $P_-^2 = P_-$.

Let us now show that $P_+ + P_- = I_2$, which proves the relations between $P_\pm$:
$$
\begin{pmatrix}
\xi_1 \xi_1^* & - \xi_1 \xi_2^* \\
- \xi_2 \xi_1^* & \xi_2 \xi_2^*
\end{pmatrix}
+
\begin{pmatrix}
\sigma(\xi_2^* \xi_2) & \sigma(\xi_2^* \xi_1) \\
\sigma(\xi_1^* \xi_2) & \sigma(\xi_1^* \xi_1) \\
\end{pmatrix}
=
\begin{pmatrix}
\xi_1 \xi_1^* + \xi_2 \xi_2^* & - \xi_1 \xi_2^* + \xi_1 \xi_2^* \\
- \xi_2 \xi_1^* + \xi_2 \xi_1^* & \xi_2 \xi_2^* + \xi_1 \xi_1^*
\end{pmatrix}
= I_2,
$$
through a systematic use of lemma \ref{Lem:RelCommSigma}. Finally, developping
$$ (M_\pm Q_\pm - M_\pm)^*(M_\pm Q_\pm - M_\pm) = (Q_\pm M_\pm^*  - M_\pm^*)(M_\pm Q_\pm - M_\pm)
= 0$$
ensures that $M_\pm Q_\pm = M_\pm$. Now $P_\pm M_\pm = M_\pm M_\pm^* M_\pm = M_\pm Q_\pm = M_\pm$, and 
\begin{align*}
M_\pm M_\mp^* &= M_\pm Q_\pm Q_\mp M_\mp^* = 0 &  M_\pm^* M_\mp &= M_\pm^* P_\pm P_\mp M_\mp= 0. 
\end{align*}
\end{proof}

The definition of odd pairings can be found in \cite{NCG} III.3 proposition 3:
\begin{definition}[Chern-Connes pairings]
The following formula defines a bilinear pairing between $K_1(\Aa)$ and $HC^{n}(\Aa)$:
\begin{equation}
\label{Eqn:DefAppariement}
\langle [U],[\phi]\rangle  =\frac{2^{-n}}{\sqrt{2 i}} \, \Gamma \left(\frac{n}{2} +1 \right)^{-1}  (\phi \# \tr)(U^* -1, U-1, U^* -1, \ldots , U -1)
\end{equation}
where $n = 2 m +1$, $[U] \in K_1(\Aa)$ and $ \phi \in ZC^{2 m +1}(\Aa)$. These pairings satisfy $\langle [U], [S \phi] \rangle = \langle [U], [\phi] \rangle $.
\end{definition}

\begin{thm}
Using the previous unitaries $U_i$, the values of the pairings are given by the table:
\begin{equation}
\label{Eqn:TableauResultats}
\begin{array}{l|c|c|c|c|}
\cline{2-5}
	 & \varphi_1 & \varphi_{2} & \varphi_{3} & \varphi_{1,2,3}\\
	\hline 
\multicolumn{1}{|l|}{[U_1]} & -\sqrt{i 2 \pi} & 0 & 0 & 0\\
   \hline
\multicolumn{1}{|l|}{[U_2]} & 0 &  -\sqrt{i 2 \pi} & 0 & 0\\
\hline
\multicolumn{1}{|l|}{[U_3]} & -\sqrt{i 2 \pi} \, 2 c \nu &  \sqrt{i 2 \pi} \, 2 c \mu & 0  & -(i 2 \pi)^{3/2}  c/3 \\
\hline
\end{array}
\end{equation}
\end{thm}
The proof of these results will occupy the next two sections.

\begin{prop}
The pairings with $U_1$ and $U_2$ always vanish, except
\begin{align*}
\langle [U_1], [\varphi_1] \rangle &= - \sqrt{i 2 \pi}
&
\langle [U_2], [\varphi_2] \rangle &= - \sqrt{i 2 \pi}.
\end{align*}
\end{prop}

\begin{proof}
For $i = 1, 2$ and $j = 1,2, 3$, let us evaluate:
$$
\langle [U_i],[\varphi_j]\rangle  = \frac{1}{\sqrt{2 i \pi}}\tau \big( (U_i^* -1) \partial _j(U_i -1) \big) = \frac{1}{\sqrt{2 i \pi}}\tau \big( U_i^* \partial _j(U_i) \big),
$$
using $\tau(\partial _j(X)) = 0$ and $\partial_j(1) = 0$. The explicit expressions of $\partial _j$ in degree $0$ and $U_i$ yields:
$$
\langle [U_i],[\varphi_j]\rangle  = -\delta_{i,j} \sqrt{2 i \pi}.
$$
The pairings with $\varphi_{1,2,3}$ are given by:
$$ \langle [U_i], \varphi_{1,2,3} \rangle = \sum_{\sigma \in \Sigma_3} \varepsilon(\sigma) \tau \Big ( (U_i^* -1) \partial_{\sigma(1)} (U_i - 1) \partial_{\sigma(2)}(U_i^* - 1) \partial_{\sigma(3)} (U_i - 1) \Big ) .$$
Each term of the above sum contains a vanishing derivation, hence $\langle [U_i], \varphi_{1,2,3} \rangle =0$.
\end{proof}

To evaluate the other odd pairings, we need to estimate the commutation relations between $\alpha$ and the derivations $\partial _i$:
\begin{notation}
We use the notation $\partial_{(u,v,w)} = u \partial_1 + v \partial_2 + w \partial_3$.
\end{notation}

\begin{prop}
\label{Prop:TransDer}
The equality
$$ \alpha_{r,s,t}\big(\partial_{(u, v, w)}(F)\big)  = \partial_{(u', v', w')}\big(\alpha_{r,s,t}(F)\big)  $$
is fulfilled if and only if
\begin{align}
\label{Eqn:TransDer}
u &= u' & v &= v' &
 w &= w' + c (v' r - s u' ).
\end{align}
\end{prop}

\begin{proof}
This is a straightforward computation.
\end{proof}

\begin{prop}
The pairings of $U_3$ with the degree $1$ cocycles are
\begin{align*}
\langle [U_3], [\varphi_1] \rangle &= -\sqrt{i 2 \pi} \, 2 c \nu
&
\langle [U_3], [\varphi_2] \rangle &= \sqrt{i 2 \pi} \, 2 c \mu
&
\langle [U_3], [\varphi_3] \rangle &= 0.
\end{align*}
\end{prop}

\begin{proof}
To simplify notations, we write $\partial $ and $\varphi$ instead of $\partial _i$ and $\varphi_i$. Remember that $U_3 = M_+ + M_-$:
$$ \sqrt{2i \pi} \langle U_3, \varphi \rangle = \tau((M_+^* + M_-^* - 1) \partial (M_+ + M_-)) = \tau((M_+^* + M_-^*) \partial (M_+ + M_-)) $$
because $\tau(\partial (M_+ + M_-)) = 0$. The trace vanishes on nonzero degree elements, hence we keep only the degree $0$ terms of the expression:
$$ \sqrt{2i \pi} \langle  U_3, \varphi \rangle =\tau(M_+^* \partial M_+ + M_-^* \partial M_-) .$$
Explicit computations yield:
\begin{align*}
M_+^* \partial M_+ &= 
\begin{pmatrix}
\xi_1^* \partial \xi_1 + \xi_2^* \partial \xi_2 & 0 \\
0 & 0
\end{pmatrix}
&
M_-^* \partial M_- &= 
\begin{pmatrix}
0 & 0 \\
0 & \sigma(\xi_2) \partial \sigma(\xi_2^*) + \sigma(\xi_1) \partial \sigma(\xi_1^*)
\end{pmatrix}.
\end{align*}
The commutation relation $ \partial \sigma(F) = \sigma( \partial (F) + k \partial _3(F)) $
holds for all $\partial = \partial _i, i \in \{1,2,3\}$, using different constants $k =k_i$. Proposition \ref{Prop:TransDer} ensures
\begin{align*}
k_1 &= - 2 c \nu & k_2 &= 2 c \mu & k_3 &= 0.
\end{align*}
Upon integrating the commutations relations in the trace of $M_+^* \partial M_+$:
\begin{multline*}
\sigma(\xi_1) \sigma( \partial \xi_1^* + k \partial _3(\xi_1^*)) + \sigma(\xi_2)  \sigma( \partial \xi_2^* + k \partial _3(\xi_2^*)) =\\
= \sigma \big ( \xi_1 \partial(\xi_1^*) +  \xi_2 \partial (\xi_2^*)+ k (\xi_1\partial _3(\xi_1^*) + \xi_2 \partial _3(\xi_2^*)) \big ).
\end{multline*}
Yet $\partial _3(\xi_i^*) = i 2 \pi \xi_i^* $ and $\xi_1 \xi_1^* + \xi_2 \xi_2^* = 1$, so:
\begin{multline*}
\sqrt{2i \pi}\langle U_3, \varphi \rangle = \tau(\xi_1^* \partial \xi_1 + \xi_2^* \partial \xi_2)  +\tau \big( \sigma(\xi_1 \partial \xi_1^* + \xi_2 \partial \xi_2^*) \big) + i 2 \pi k\\
= \tau \big( \xi_1^* \partial \xi_1 + \xi_2^* \partial \xi_2  + \sigma(\xi_1 \partial \xi_1^* + \xi_2 \partial \xi_2^*)\big) + i 2 \pi k  \\
= \tau\big( \xi_1^* \partial \xi_1 + \xi_2^* \partial \xi_2  + \partial (\xi_1^*) \xi_1  + \partial (\xi_2^*) \xi_2\big) + i 2 \pi k  \\
= \tau \big( \partial (\xi_1^* \xi_1 + \xi_2^* \xi_2) \big) + i 2 \pi k = \tau( \partial (1) ) + i 2 \pi k,
\end{multline*}
where we applied lemma \ref{Lem:RelCommSigma}. Taking into account the different values of $k_i$:
\begin{align*}
\langle  U_3, \varphi_1 \rangle &=   -\sqrt{2i \pi} \; 2 c \nu
&
\langle U_3,  \varphi_2\rangle &=   \sqrt{2i \pi} \; 2 c \mu
&
\langle U_3,  \varphi_3 \rangle &=   0.
\end{align*}
\end{proof}

\section{Top Degree Pairing}

\begin{prop}
The pairing $\langle [U_3] , \varphi_{1,2,3}\rangle$ is 
$$ \langle [U_3], [\varphi_{1,2,3}] \rangle = -(i 2 \pi)^{3/2} c /3 .$$
\end{prop}
The proof of this proposition will fill the remainder of this section. Setting $\Uu = U_3 = M_+ + M_-$, we want to evaluate:
\begin{equation}
\label{Eqn:DefPhi123}
\langle [U_3], \varphi_{1,2,3}  \rangle = \frac{1}{8 \sqrt{2 i}} \frac{4}{3 \sqrt{\pi}} \sum_{\sigma \in \Sigma_3} \varepsilon(\sigma) T_{\sigma(1) \sigma(2) \sigma(3)}
\end{equation}
where $\varepsilon(\sigma)$ is the signature of the permutation $\sigma$ of $\{1,2,3 \}$ and
\begin{align*}
T_{ijk} &= \tau \big( (\Uu^*-1) \partial_{i} (\Uu-1) \partial_{j} (\Uu^*-1) \partial_{k} (\Uu-1)  \big) \\
&= \tau \big( (\Uu^* -1) \partial_{i} (\Uu) \partial_{j} (\Uu^*) \partial_{k} (\Uu)  \big).
\end{align*}
All the following terms have odd degree -- and therefore vanishing trace:
$$ \tau \big( \partial _i(\Uu)\partial _j(\Uu^*)\partial _k(\Uu) \big) =  \tau \big (\partial_{i} (M_+ + M_-) \partial_{j} (M_+^* + M_-^*) \partial_{k} (M_+ + M_-) \big ) = 0 .$$
Hence 
$$ T_{ijk} = \tau \big( \Uu^* \partial_{i} (\Uu) \partial_{j} (\Uu^*) \partial_{k} (\Uu)  \big). $$
\begin{lem}
\label{Lem:Relations1}
The following relations hold:
\begin{align}
\label{Eqn:Relation11}
\partial _i(M_\pm) M^*_\mp &= - M_\pm \partial _i(M^*_\mp) & \partial _i(M_\pm^*) P_\mp &= - M_\pm^* \partial _i(P_\mp)
\end{align}
\begin{equation}
\label{Eqn:Relation12}
\tau\big(P_\mp \partial _i(M_\pm) \partial _j(M_\pm^*)\big) = \tau\big(P_\pm \partial _j(P_\pm) \partial _i(P_\pm)\big) 
\end{equation}
\begin{equation}
\label{Eqn:Relation13}
\tau\big(P_\pm \partial _i (M_\pm) \partial _j(M_\pm^*)\big) = \tau \big( \partial _i(M_\pm) \partial_j(M_\pm^*) + P_\pm \partial _j(P_\pm) \partial _i(P_\pm) \big).
\end{equation}
\end{lem}

Beware of permutations between $i$ and $j$ in the equations \eqref{Eqn:Relation12} and \eqref{Eqn:Relation13} !

\begin{proof}
The first series can be proved by integration by parts:
$$ \partial _i(M_\pm) M^*_\mp = \partial _i(M_\pm M^*_\mp) - M_\pm \partial _i(M^*_\mp) =  -M_\pm \partial _i(M^*_\mp) $$
because (lemma \ref{Prop:RelUnit}) $M_\pm M^*_\mp = 0$. We also have $M^*_\pm P_\mp = M^*_\pm P_\pm P_\mp = 0$, which enables to prove the second equality of \eqref{Eqn:Relation11} using the same method. Now
\begin{align*}
\tau(P_\mp \partial _i(M_\pm)& \partial _j(M_\pm^*)) = \tau(\partial _j(M_\pm^*) P_\mp \partial _i(M_\pm) ) =\\
&= - \tau \big( M_\pm^* \partial _j(P_\mp) \partial _i(M_\pm) \big) = \tau \big( M_\pm^* \partial _j(P_\pm) \partial _i(M_\pm) \big)\\
&= \tau \Big( M_\pm^* \partial _j(P_\pm) \big( \partial _i(P_\pm) M_\pm + P_\pm \partial _i( M_\pm)  \big) \Big) =\\
&= \tau \Big( P_\pm \partial _j(P_\pm) \partial _i(P_\pm)  + M_\pm^* P_\pm \partial _j(P_\pm) P_\pm \partial _i(M_\pm) \Big) =\\
&= \tau(P_\pm \partial _j(P_\pm) \partial _i(P_\pm)),
\end{align*}
using \eqref{Eqn:Relation11}, $\partial _i(P_\mp) = - \partial _i(P_\pm)$ and $P_\pm \partial _i(P_\pm) P_\pm = 0$. This last equality is true because $P_\pm$ is an idempotent and $\partial _i$ a derivation. 

Regarding the last equation:
\begin{multline*}
\tau \big( P_\pm \partial _i(M_\pm) \partial _j(M_\pm^* ) \big) = \tau \big( (\partial _i(P_\pm M_\pm) - \partial _i(P_\pm) M_\pm) \partial _j(M_\pm^*) \big)=\\
= \tau \big( \partial _i(M_\pm) \partial_j(M_\pm^*) - \partial _i(P_\pm)  M_\pm \partial _j(M^*_\pm)P_\pm + \partial _i(P_\pm) M_\pm M_\pm^* \partial _j(P_\pm) \big) =\\
= \tau \big( \partial _i(M_\pm) \partial_j(M_\pm^*) - P_\pm \partial _i(P_\pm) P_\pm M_\pm \partial _j(M^*_\pm) + P_\pm \partial _j(P_\pm) \partial _i(P_\pm) \big) =\\
= \tau \big( \partial _i(M_\pm) \partial_j(M_\pm^*) + P_\pm \partial _j(P_\pm) \partial _i(P_\pm) \big),
\end{multline*}
possible because $M_\pm^* = M_\pm^* P_\pm$.
\end{proof}

We will also need the following lemma:
\begin{lem}
\label{Lem:Relations2}
If $(i,j) = (1,2)$ or $(2,1)$, we have both
$$ \tau( \partial _i(\Uu) \partial _j(\Uu^*) ) = \tau \big( \partial _i(M_+) \partial_j(M_+^*) + \partial _i(M_-) \partial_j(M_-^*) \big) $$
and
$$ \tau \big( \partial _1(\Uu) \partial _2(\Uu^*) - \partial _2(\Uu) \partial _1(\Uu^*) \big) = 0 .$$
\end{lem}


\begin{proof}
The first equality is obvious by keeping only the degree 0 terms in
$$ \tau( \partial _i(\Uu) \partial _j(\Uu^*) ) =\tau \big( \partial _i(M_+ + M_-) \partial _j(M_+^* + M_-^*) \big). $$
The second one is obtained by integration by parts:
\begin{multline*}
 \tau \big( \partial _1(\Uu) \partial _2(\Uu^*) - \partial _2(\Uu) \partial _1(\Uu^*) \big) = \\
=\tau \Big( \partial _1 \big( \Uu \partial _2(\Uu^*) \big) - \Uu \partial _1 \partial _2(\Uu^*) - \partial _2 \big( \Uu \partial _1(\Uu^*) \big) + \Uu \partial _2 \partial _1(\Uu^*) \Big)=\\
= c \tau \big( \Uu \partial _3(\Uu^*) \big) = 0,
\end{multline*}
using $[\partial _1, \partial _2] = - c \partial _3$ and $\langle \varphi_3, \Uu \rangle = 0$.
\end{proof}

\subsubsection{Terms $T_{132}$ and $T_{231}$}

Using the identity $\partial _3(M_\pm^*) = \pm i 2\pi M_\pm^*$, we can evaluate $T_{231}$:
\begin{align*}
T_{231} = i 2 \pi \tau \big (& ( M_+^* + M_-^*)\partial_{2} (\Uu) (M_+^* - M_-^*) \partial_{1} (\Uu) \big ) \\
= i 2 \pi \tau \Big (&   M_+^*\partial_{2} (\Uu) M_+^* \partial_{1} (\Uu) - M_+^*\partial_{2} (\Uu) M_-^* \partial_{1} (\Uu) \\
&+  M_-^*\partial_{2} (\Uu) M_+^* \partial_{1} (\Uu) -  M_-^*\partial_{2} (\Uu) M_-^* \partial_{1} (\Uu) \Big ).
\end{align*}
Same thing for $T_{132}$:
\begin{align*}
T_{132} = i 2 \pi \tau \big (& ( M_+^* + M_-^*)\partial_{1} (\Uu) (M_+^* - M_-^*) \partial_{2} (\Uu) \big ) \\
= i 2 \pi \tau \Big (&   M_+^*\partial_{1} (\Uu) M_+^* \partial_{2} (\Uu) - M_+^*\partial_{1} (\Uu) M_-^* \partial_{2} (\Uu) \\
&+  M_-^*\partial_{1} (\Uu) M_+^* \partial_{2} (\Uu) -  M_-^*\partial_{1} (\Uu) M_-^* \partial_{2} (\Uu) \Big ).
\end{align*}
Taking into account the $\varepsilon(\sigma)$ of \eqref{Eqn:DefPhi123} and the fact that $\tau$ is a trace,
$$
T_{231} - T_{132}  = i 4 \pi \tau \Big ( M_+^*\partial_{2} (\Uu) M_-^* \partial_{1} (\Uu) -  M_-^*\partial_{2} (\Uu) M_+^* \partial_{1} (\Uu) \Big ).
$$
As $\Uu = M_+ + M_-$, we can put both terms in the general form:
$$ \tau \big ( M_+^*\partial_{i} (M_+ + M_-) M_-^* \partial_{j} (M_+ + M_-) \big ). $$
Keeping only elements of total degree $0$:
\begin{align*}
\tau \big (& M_+^*\partial_{i} (M_+ ) M_-^* \partial_{j} (M_-) + M_+^*\partial_{i} (M_- ) M_-^* \partial_{j} (M_+) \big ) 
\end{align*}
The first term integrates by parts:
$$ \big( \partial _i(M_+^* M_+) - \partial _i(M_+^*) M_+ \big) M_- ^* \partial _j(M_-). $$
It vanishes because $\partial _i(M_+^* M_+) = \partial_i (Q_+) = 0$ and $M_+ M_-^* = 0$. For the second term,
\begin{multline*}
\tau \Big( M_+^* \big( \partial _i(\underbrace{M_- M_-^*}_{= P_-}) - M_- \partial _i(M_-^*) \big) \partial _j(M_+) \Big)  = \tau \big( M_+^* \partial _i(P_-) \partial _j(M_+) \big) =\\
= -\tau \big( \partial _i(M_+^*) P_- \partial _j(M_+) \big) = -\tau \big( P_- \partial _j(M_+)  \partial _i(M_+^*) \big) =- \tau(P_+ \partial _i(P_+) \partial _j(P_+))
\end{multline*}
applying \eqref{Eqn:Relation11}, trace property and then \eqref{Eqn:Relation12}.

Therefore, the contribution of $T_{132}$ and $T_{231}$ is
$$ T_{231} - T_{132} = - i 4 \pi \tau \big( P_+ (\partial _2 P_+ \partial _1 P_+ - \partial _1 P_+ \partial _2 P_+) \big) .$$

\subsubsection{Terms $T_{123}$ and $T_{213}$}

As $\partial _3(M_\pm) = \mp i 2\pi M_\pm$, the terms $T_{123}$ and $T_{213}$ can be written:
\begin{multline*}
T_{ij3} =  i 2 \pi \tau \big( ( M_+^* + M_-^*) \partial_{i} (\Uu) \partial_{j} (\Uu^*) (M_+ - M_-) \big) =\\
=  i 2 \pi \tau \big( ( P_+ - P_-) \partial_{i} (\Uu) \partial_{j} (\Uu^*) \big) =  i 2 \pi  \tau\big( (2 P_+ - I_2) \partial _i(\Uu) \partial _j(\Uu^*) \big).
\end{multline*}
Taking the difference $T_{123} - T_{213}$, then using lemma \ref{Lem:Relations2}, the expression becomes:
$$ T_{123} - T_{213} = - i 4 \pi \tau \Big( P_+ \big(\partial_1 (\Uu) \partial _2(\Uu^*)- \partial_2 (\Uu) \partial _1(\Uu^*) \big) \Big) .$$
Let us study the term
\begin{multline*}
\tau \big(  P_+ \partial_{i} (M_+ + M_-) \partial_{j} (M_+^* + M_-^*) \big) = \tau \Big( P_+ \big( \partial_{i} (M_+) \partial _j(M_+^*) + \partial _i (M_-) \partial_{j} (M_-^*) \big) \Big) \\
= \tau \big( \partial _i(M_+) \partial _j(M_+) + P_+ \partial _j(P_+) \partial _i(P_+) + P_- \partial _j(P_-) \partial _i(P_-) \big) =\\
=  \tau \big( \partial _i(M_+) \partial _j(M_+^*) + \partial _j(P_+) \partial _i(P_+) \big),
\end{multline*}
using the equations \eqref{Eqn:Relation13} and \eqref{Eqn:Relation12}, and then $\partial _i(P_-) \partial _j(P_-) = \partial _i(P_+) \partial _j(P_+)$. Next,
$$\tau(\partial _1(P_+) \partial _2(P_+) - \partial _2(P_+) \partial _1(P_+) ) = 0, $$
because $\tau$ is a trace.

Finally, the contribution of the terms $T_{123}$ and $T_{213}$ is
$$
T_{123} - T_{213} = -i 4 \pi \tau \Big( \partial _1(M_+) \partial _2(M_+^*) - \partial _2(M_+) \partial _1(M_+^*) \Big).
$$

\subsubsection{Terms $T_{312}$ and $T_{321}$}
\begin{lem}
The analogs of \ref{Lem:Relations1} hold.
\begin{align}
\label{Eqn:Relation21}
Q_\mp \partial _i M^*_\pm &= -\partial _i(Q_\mp) M_\pm^* = 0
&
Q_\pm \partial _i M^*_\pm &= \partial _i(M^*_\pm)
\end{align}
\begin{equation}
\label{Eqn:Relation22}
Q_\mp \partial_{i} (M_\pm^*) \partial _j(M_\pm) = 0
\end{equation}
\begin{equation}
\label{Eqn:Relation23}
Q_\pm \partial_{i} (M_\pm^*) \partial _j(M_\pm) = \partial _i (M^*_\pm) \partial _j(M_\pm).
\end{equation}
\end{lem}

\begin{proof}
For the first series of relations,
$$ Q_\mp \partial _i M^*_\pm = \partial _i(Q_\mp M_\pm^*)-\partial _i(Q_\mp) M_\pm^* = 0$$
because $Q_\mp M_\pm^* = Q_\mp Q_\pm M_\pm^* = 0$ and $\partial_i (Q_\pm) = 0$. Likewise,
$$ Q_\pm \partial _i M^*_\pm = \partial _i(Q_\pm M^*_\pm) - \partial _i(Q_\pm) M^*_\pm = \partial _i(M^*_\pm),
$$
thanks to $Q_\pm M^*_\pm = M_\pm^*$. For the second series,
$$
Q_\mp \partial _i(M_\pm^*) \partial _j(M_\pm) = - \partial _i(Q_\mp) M_\pm^* \partial _j(M_\pm)  = 0,
$$
using the previous relations. The third relation is obvious starting from \ref{Eqn:Relation21}.
\end{proof}

The terms $T_{312}$ and $T_{321}$ can be written:
\begin{multline*}
T_{3ij} =  i 2 \pi \tau \big( ( M_+^* + M_-^*) (M_+ - M_-) \partial_{i} (\Uu^*) \partial_{j} (\Uu) \big) = \\
= i 2 \pi \tau \big( ( Q_+ - Q_-) \partial_{i} (\Uu^*) \partial_{j} (\Uu) \big) = i 2 \pi \tau\big( ( 2 Q_+ - I_2) \partial _i(\Uu^*) \partial _j(\Uu) \big).
\end{multline*}
Taking the difference $T_{312} - T_{321}$, then using lemma \ref{Lem:Relations2}, we get:
$$ T_{312} - T_{321} =- i4 \pi \tau \Big( Q_+ \big(\partial_1 (\Uu^*) \partial _2(\Uu)- \partial_2(\Uu^*) \partial _1(\Uu) \big) \Big) .$$
Let us study the following term
\begin{multline*}
\tau \big(  Q_+ \partial_{i} (M_+^* + M_-^*) \partial_{j} (M_+ + M_-) \big) =\\
= \tau \big( Q_+ ( \partial_{i} (M_+^*) \partial _j(M_+) + \partial _i (M_-^*) \partial_{j} (M_-) \big)= \tau \big( \partial _i(M_+^*) \partial _j(M_+) \big).
\end{multline*}
The difference can thus be written:
$$ T_{312} - T_{321} =- i4 \pi \tau \Big(\partial _1(M_+^*) \partial _2(M_+) - \partial _2(M_+^*) \partial _1(M_+) \Big) .$$

\subsubsection{Synthesis and Final Computation}

Forming the synthesis of the studied terms:
\begin{align*}
6 \sqrt{2 i \pi} \; \langle [U_3], \varphi_{1,2,3} \rangle &= \big( T_{231} - T_{132} + T_{123}  - T_{213} + T_{312}  - T_{321} \big) \\
&=-  i 4 \pi \tau \big( P_+ (\partial _1 P_+ \partial _2 P_+ - \partial _2 P_+ \partial _1 P_+) \big) \\
&\phantom{=\;}-  i 4 \pi \tau \Big( \partial _1(M_+) \partial _2(M_+^*) - \partial _2(M_+) \partial _1(M_+^*) \Big) \\
&\phantom{=\;} - i4 \pi \tau \Big(\partial _1(M_+^*) \partial _2(M_+) - \partial _2(M_+^*) \partial _1(M_+) \Big)\\
&= - i 4 \pi \tau \big( P_+ (\partial _1 P_+ \partial _2 P_+ - \partial _2 P_+ \partial _1 P_+) \big)\\
&= - i 4 \pi \left \langle [P_+], [\psi_{1,2}] \right \rangle ,
\end{align*}
where $\psi_{1,2}$ is the cyclic cocycle defined on $C^\infty(T^2)$ by 
$$
(a_0,a_1, a_2) \longmapsto \int_{T^2} a_0 \big( \partial _1 a_1 \, \partial _2 a_2 - \partial _2 a_1 \, \partial _1 a_2 \big)(x,y) dx dy.
$$
The right-hand side is a pairing on the algebra $A = C(T^2)$. We can compute it using a connection.

\bigbreak

First identify the module of $C(T^2)$ which corresponds to the projector $P_+$.

\begin{lem}
\label{Lem:IsomPP}
The module $P_+ A^2$ is isomorphic to $(\ModQHM[-c][0,0])_A$.
\end{lem}

Notice that we only need an identification as a \emph{module}, and not as a \emph{bimodule}.

\begin{notation}
We denote the elements of $\ModQHM[-c][0,0]$ by $\zeta^*$. More generally, we include $\ModQHM[-c][0,0]$ into $\QHM[c][0,0]$ in the following computation.
\end{notation}

\begin{proof}
To identify the module $M$ associated to $P_+$, the simplest thing is to interpret $M_+^*$ and $M_+$ respectively as maps $P_+ A^2 \to M$ and $M \to P_+ A^2$. Formally, we introduce the maps $\Phi \colon P_+ A^2 \to \ModQHM[-c][0,0]$ and $\Psi \colon\ModQHM[-c][0,0] \to P_+ A^2$:
\begin{align*}
\Phi %
\begin{pmatrix}
a_1 \\
a_2
\end{pmatrix} &= \begin{pmatrix}
\xi_1^* & - \xi_2^*
\end{pmatrix}
\begin{pmatrix}
a_1 \\
a_2
\end{pmatrix}
&
\Psi( \zeta^*) &= \begin{pmatrix}
\xi_1 \\
-\xi_2
\end{pmatrix} \zeta^*. 
\end{align*}
Using the properties of $M_+$ and $M_+^*$, we  see that the maps $\Phi$ and $\Psi$ are inverse to one another, and that they preserve the scalar products.
\end{proof}

The following result is well known, and can be obtained by using connections:
\begin{prop}
The pairing of $\ModQHM[-c]$ with $\psi_{1,2}$ is given by:
$$ \langle [\ModQHM[-c][0,0] ] , [\psi_{1,2} ] \rangle = i 2 \pi c  .$$
\end{prop}

This proposition enables us to complete the computation of $\langle [U_3], [\varphi_{1,2,3}] \rangle $:
$$ \langle [U_3], [\varphi_{1,2,3}] \rangle =  \frac{- i 4 \pi}{6 \sqrt{i 2 \pi}} i 2 \pi c = -c \frac{(i 2 \pi)^{3/2}}{3} .$$

\section{Consequences of the Pairings}

An immediate consequence of \eqref{Eqn:KThQHM} and table \eqref{Eqn:TableauResultats} is that $(U_1, U_2,U_3)$ form a basis of $K_1(\QHM) \otimes \C$. It is therefore natural to wonder if these elements generate $K_1(\QHM)$. The answer is no.

\smallbreak

To prove the above, first notice that since the QHM are Pimsner algebras associated to the Hilbert bimodule $\ModQHM[c]$, their $K$-theory fits in the 6-terms exact sequence (see \cite{PiCrG}): 
\begin{equation}
\label{Eqn:6Terms}
\xymatrix@C=1.5cm{
K_0(A) \ar[r]^{\id -  [\ModQHM[c]]} & K_0(A) \ar[r]^{\iota} & K_0(\QHM) \ar[d]^{\partial } \\
K_1(\QHM) \ar[u]^\partial  & K_1(A) \ar[l]^{\iota} & K_1(A) \ar[l]^{\id - [\ModQHM[c]]} \\
}
\end{equation}
In the diagram \eqref{Eqn:6Terms} above, we use the right Kasparov product with $\id -  [\ModQHM[c]]$, where $[\ModQHM[c]]$ is the Kasparov module $(\ModQHM[c], \phi, 0) \in KK(A, A)$ and $\phi$ is the left action of $A$ on $\ModQHM[c]$. To establish the previous exact sequence, Pimsner's approach is in fact to start with a short exact sequence
\begin{equation}
\label{Eqn:SEToep}
0 \to J \to \ToepQHM \to \QHM \to 0
\end{equation}
and to prove that both $J$ and $\ToepQHM $ are $KK$-equivalent to $A$. The algebra $\ToepQHM $ was defined in \cite{PiCrG}. In our special case, one proves that $\ToepQHM $ can be identified with a $C^*$-subalgebra of the tensor product $\Tt \otimes \QHM$ of $\QHM$ with the Toeplitz algebra $\Tt$. We identify $\Tt$ with the unital $C^*$-algebra generated by the projector $P$ and the isometry $S$ with the relations
\begin{align*}
S S^* &= 1 - P
&
S^* S &= 1.
\end{align*}
The realisation of $\ToepQHM $ inside $\Tt \otimes \QHM$ is generated by $1 \otimes a$ and $S \otimes \xi$ for $a \in A$ and $\xi \in \ModQHM$. Notice that $\Tt $ is a nuclear $C^*$-algebra, thus we do not need to specify which $C^*$-norm we are using in the tensor product.

To prove that $(U_1, U_2, U_3)$ does not generate $K_1(\QHM)$, we study their images by the index map $\partial \colon  K_1(\QHM) \to K_0(A)$.
\begin{prop}
\label{Prop:Indice}
The index map of the $U_i$ is given by:
\begin{align*}
\partial (U_1) &= 0 
&
\partial (U_2) &= 0 
&
\partial (U_3) &= [Q_-] \ominus [P_+],
\end{align*}
with the notation of proposition \ref{Prop:RelUnit}.
\end{prop}

\begin{proof}
Since $U_1$ and $U_2$ come from the inclusion $A \inj \QHM$, the first two equalities are clear.

Regarding the last equality, we are really going to compute the index map in the 6-term exact sequence associated to \eqref{Eqn:SEToep} and then translate the result in terms of $K$-theory of $A$.

It is readily checked that 
$$
\U_3 =
\begin{pmatrix}
S \otimes M_+ + S^* \otimes M_- & P \otimes P_+ \\
P \otimes Q_- & S^* \otimes M_+^* + S \otimes M_-^*
\end{pmatrix}
$$
is a unitary lift of $U_3$ in $\ToepQHM $. It is then easy to complete the index computation:
\begin{multline*}
\U \begin{pmatrix}
1_2 & 0 \\ 0 & 0 
\end{pmatrix}
\U^*
=\\
=
\begin{pmatrix}
S \otimes M_+ + S^* \otimes M_- & 0 \\ 
P \otimes Q_- & 0
\end{pmatrix}
\begin{pmatrix}
S^* \otimes M_+^* + S \otimes M_-^* & P \otimes Q_-\\
0 & 0
\end{pmatrix}= \\
=\begin{pmatrix}
1 \otimes 1 - P \otimes P_+  & 0 \\
0 & P \otimes Q_-
\end{pmatrix}.
\end{multline*}
Hence, the index of $U_3$ in $J$ is $P \otimes Q_- - P \otimes P_+$, which is the image of $Q_- \ominus P_+$ in $J$. The $KK$-equivalence between $A$ and $J$ then entails $\partial (U_3) = Q_- \ominus P_+$.
\end{proof}

In \cite{RieffelCancelThm}, Rieffel has given an explicit description of the finite projective modules over $C(T^2)$. He proves that $K_0(A)$ can be identified with the pairs $(d,t)$ where $d$ (dimension) and $t$ (twist) are in $\Z$. We can choose the definition of the twist such that $[\ModQHM[c]] \simeq (1, c)$, and thus $[P_+] = (1, -c)$. 

From this identification, together with the proposition 3.10 of \cite{RieffelCancelThm}, we see that 
$$ \id - [\ModQHM] \colon (d,t) \mapsto (d,t) - (d, t + c d) = (0, - c d). $$
Thus the kernel of $\id - [\ModQHM ]$ is generated by $(0,1)$. As $\partial (U_3) = [Q_-] \ominus [P_+] \simeq (0, c)$, \emph{$(U_1, U_2,U_3)$ is not a generator of $K_1(\QHM)$}.

\medbreak

Another consequence of proposition \ref{Prop:Indice} and the previous computation is
\begin{cor}
For any $U \in K_1(\QHM)$, 
\begin{equation}
\label{Eqn:Trans}
 \langle [U], \varphi_{1,2,3} \rangle = \frac{-\sqrt{i 2 \pi}}{3} \langle \partial [U], \psi_{1,2} \rangle .
\end{equation}
\end{cor}

\begin{proof}
Since the pairing is bilinear, we only have to check this on a basis of $K_1(\QHM) \otimes \C$. The equality is true for $U = U_1$ and $U = U_2$, since $\partial (U_i) = 0$.

For $U_3$, notice that $\langle [Q_-], \varphi_{1,2,3} \rangle = \langle [1], \varphi_{1,2,3} \rangle = 0$. Hence
\begin{equation*}
\langle [U_3], \varphi_{1,2,3} \rangle = \frac{-i 4 \pi}{6 \sqrt{i 2 \pi}} \langle -[P_+], \psi_{1,2} \rangle  = \frac{-\sqrt{i 2 \pi}}{3} \langle \partial (U_3), \psi_{1,2} \rangle.
\end{equation*}
\end{proof}

The formula \eqref{Eqn:Trans} enables us to prove further properties of the pairing with $\varphi_{1,2,3}$:
\begin{cor}
There is a $K$-homology element $K' \in K^1(\QHM)$ such that for any $[U] \in K_1(\QHM)$,
$$ \langle [U], \varphi_{1,2,3} \rangle = \frac{(i 2 \pi)^{3/2}}{3} \langle [U], K' \rangle .$$
\end{cor}

In particular, the pairing between $K$-theory and $K$-homology on the right hand side takes only integer values, as one can check by direct examination of \eqref{Eqn:TableauResultats}.

\begin{proof}
We know that in $A = C(T^2)$, one can find a $K$-homology element $K$ such that 
$$ \langle [E], \psi_{1,2} \rangle =- i 2 \pi \langle [E], K \rangle $$
for any $[E] \in K_0(A)$. It is known that in the 6-terms exact sequence \eqref{Eqn:6Terms}, one can see the boundary maps as multiplication by some element $\delta \in KK^1(D,A)$. Multiplying $K \in K^0(A) = KK_0(A,\C)$ by $\delta$, we get $K' \in K^1(\QHM)$ such that
$$ i 2 \pi \langle [U], K' \rangle = i 2 \pi \langle \partial [U], K \rangle = - \langle \partial [U], \phi_{1,2} \rangle = \frac{3}{\sqrt{i 2 \pi}} \langle [U], \varphi_{1,2,3} \rangle .$$
\end{proof}

\section{Dimension of $HP^*(\QHMl)$}

To complete the description of $HP^*(\QHMl)$, it only remains to be proven that there are no extra cyclic cocycles. We therefore compute the dimension of $HP^0(\QHMl)$ and $HP^1(\QHMl)$. We stick to computation of pairings and periodic theory. We also use results from \cite{ChernGrensingGabriel}: in this article, it is proved that the smooth QHM $\QHMl$ fit in a short exact sequence of locally convex algebras. 

In this section, we denote by $\Aa$ the algebra $ C^\infty(T^2)$ with its usual Fréchet algebra structure. We start with the following definitions originally due to Cuntz (see for instance \cite{WeylCuntz}, 2.2 and 2.3):
\begin{definition}[Smooth Toeplitz algebra]
The \emph{smooth compact operators} $\Kk$ are the $\N \times \N$ matrices $(a_{i,j})$ with rapidly decreasing complex entries. The topology on $\Kk$ is given by the norms:
$$ p_n\big( (a_{i,j}) \big) = \sum_{i,j} |1 + i|^n |1 + j|^n |a_{i,j}| .$$
The \emph{smooth Toeplitz algebra} $ \Toep $ is topologically the direct sum $\Toep = \Kk \oplus C^\infty(S^1)$, where $C^\infty(S^1)$ is equipped with its usual Fréchet structure. The multiplication in $\Toep $ is described \textsl{via} an action of $\Toep $ on $\Ss(\N)$ (rapidly decreasing sequences). $\Kk$ acts in the natural way, and $C^\infty(S^1)$ by truncated convolution. The function $\sum a_k z^k$ acts on $(\xi_i) \in \Ss(\N)$ by:
$$ (a \xi)_i = \sum_{k + j = i} a_k \xi_j .$$
\end{definition}
The algebra $\QHMl$ fits in the linearly split exact sequence:
\begin{equation*}
 0 \to \kerl \to \ToepQHMl \xrightarrow{\pi} \QHMl \to 0,
\end{equation*}
where $\ToepQHMl$ is the subalgebra generated in the projective tensor product $\Toep \hat{\otimes} \QHMl$ by $1 \otimes a$, $S \otimes \xi$ and $S^* \otimes \xi^*$ for $a \in \Aa$ and $\xi \in \ModQHMl$ (see \cite{ChernGrensingGabriel}). The map $\pi$ is defined on the generators by:
\begin{align*}
\pi(1 \otimes a) &= a
&
\pi(S \otimes \xi) &= \xi
&
\pi(S^* \otimes \xi^*) &= \xi^*.
\end{align*}

From this short exact sequence, the article \cite{ExcisionBivCQ} ensures that there is a 6-terms exact sequence in periodic cyclic cohomology:
\begin{equation}
\label{Eqn:6TermsHP}
\xymatrix@C=1cm{
HP^0(\kerl) \ar[d] & HP^0(\ToepQHMl) \ar[l] & HP^0(\QHMl) \ar[l]_-{\pi^*} \\
HP^1(\QHMl) \ar[r]^-{\pi^*}  & HP^1(\ToepQHMl) \ar[r] & HP^1(\kerl) \ar[u] \\
}
\end{equation}

It was also proved in \cite{ChernGrensingGabriel} that there are $kk$-equivalences $\ToepQHMl \simeq \Aa$ and $\kerl \simeq \Aa$. Consequently
\begin{align}
\label{Eqn:DimHP}
HP^0(\kerl) = HP^0(\ToepQHMl) &= HP^0(\Aa) = \C^2
\\
HP^1(\kerl) = HP^1(\ToepQHMl) &= HP^1(\Aa) = \C^2.
\end{align}
Moreover, the inclusion $\Aa \inj \ToepQHMl$ induces an isomorphism at the level of $K$-theory. Hence the unitaries $V_i = 1 \otimes U_i \in \ToepQHMl$ form a generating set of $K_1(\ToepQHMl)$, and if $E$ is a \emph{projector} in $M_n(\Aa)$ with dimension $1$ and twisting $1$, then $1 \otimes 1$ and $1 \otimes E$ are a generating set of $K_0(\ToepQHMl)$. These generating sets \emph{separate} the periodic cyclic cocycles.

\smallbreak

Let us now evaluate the image of $\pi^*$ by computing pairings on $\ToepQHMl$. In the odd case, 
$$ \langle V_i, \pi^* \varphi_j \rangle = \langle \pi_*(V_i), \varphi_j \rangle = \langle U_i, \varphi_j \rangle ,$$
which together with \eqref{Eqn:DimHP} proves that $(\pi^* \varphi_j)_{j =1,2}$ is a basis of $HP^1(\ToepQHMl)$.

In the even case, our analysis of the map $\id - [\ModQHM [c]]$ ensures that in $K_0(\QHMl)$ $\iota( E) \ominus  \iota(1)$ is a torsion element. Here we used $\iota \colon \Aa \inj \QHMl$ to distinguish $E \in K_0(\Aa)$ from $\iota(E) \in K_0(\QHMl)$. The definition of $\pi$ guaranties $\pi(1 \otimes E) = \iota(E)$ and $\pi(1 \otimes 1) = \iota(1)$. The bilinearity of the pairing therefore enforces:
$$ \langle 1 \otimes E, \pi^* \psi \rangle = \langle \iota(E) , \psi \rangle = \langle \iota(1), \psi \rangle = \langle 1 \otimes 1, \pi^* \psi \rangle  .$$
Since $1 \otimes 1$ and $1 \otimes E$ separate $HP^*(\Aa)$, we see that $ \pi^* \big( HP^0(\QHMl) \big)$ has dimension at most $1$. It is easy to use the trace to check that the image is non zero.

Putting the above results in the exact sequence \eqref{Eqn:6TermsHP}, we get:
\begin{prop}
The periodic cyclic cohomology of $\QHMl$ is given by:
\begin{align*}
HP^0(\QHMl) &= \C^3
&
HP^1(\QHMl) &= \C^3.
\end{align*}
\end{prop}

Yet, proposition \ref{Prop:CCycle} provided us with \emph{7} cyclic cocycles. The tables \eqref{Eqn:TableauResPairs} and \eqref{Eqn:TableauResultats} enable us to exhibit linearly independent families of cyclic cocycles. Notice that the pairing $\langle [K], [\varphi] \rangle $ only depends on the class of $\varphi$ in \emph{periodic cyclic cohomology} and thus the families are independent in $HP^*(\QHMl)$. Hence we get:
\begin{thm}
\label{Thm:Bases}
Taking the notations of proposition \ref{Prop:CCycle},
\begin{itemize}
\item
the family $(\tau, \varphi_{1,3}, \varphi_{2,3})$ is a basis of $HP^0(\QHMl)$;
\item
the family $(\varphi_1, \varphi_2, \varphi_{1,2,3})$ is a basis of $HP^1(\QHMl)$.
\end{itemize}
\end{thm}

It follows from \cite{ChernGrensingGabriel} that the (tensorized) Chern character is an isomorphism for QHM. As we know the $K$-theory of QHM, we get:
\begin{align*}
HP_0(\QHMl) &= K_0(\QHMl) \otimes_\Z \C = \C^3
&
HP_1(\QHMl) &= K_1(\QHMl) \otimes_\Z \C = \C^3.
\end{align*}
Since the Chern-Connes pairings factorize through the Chern-Character from $K$-theory to cyclic homology (see \cite{Loday}, section 8.3), we can describe the periodic cyclic homology. Notice that in the following we require $\mu \neq 0 \neq \nu$ so that the modules $\ModNl$ and $\ModNNl$ over $\QHMl[c][\mu,\nu]$ \emph{exist} (see definition \ref{Def:ModNl} and definition \ref{Def:ModNN}). 
\begin{thm}
\label{Thm:Bases2}
\begin{itemize}
\item
 The family $(\Ch(U_1), \Ch(U_2), \Ch(U_3))$ is a basis of $HP_1(\QHMl)$.
\item 
  If $\mu \neq 0 \neq \nu$, then  the family which consists of $\Ch\big(\QHMl[c][\mu,\nu] \big)$, $\Ch \big(\ModNl[c][\mu,\nu] \big)$ and $\Ch \big( \ModNNl[c][\mu,\nu] \big)$ is a basis of $HP_0(\QHMl[c][\mu,\nu])$.
\end{itemize}
\end{thm}

\section{Final Remarks}

%

The previous discussion of linear independence in $HP^*(\QHMl)$ or $HC^*(\QHMl)$ is relevant because the cyclic cocycle $\varphi_3$ is nonzero in $HC^1(\QHMl)$ and $(\varphi_1, \varphi_2, \varphi_3)$ is linearly independent in $HC^1(\QHMl)$, yet $[\varphi_3] = 0$ in $HP^1(\QHMl)$.

More generally, we have the following proposition which ensures the linear independence of the seven cyclic cocycles of proposition \ref{Prop:CCycle}:
\begin{prop}
\label{Prop:DualCycles}
The following Hochschild cycles are ``dual'' to the cocycles of proposition \ref{Prop:CCycle}, in the sense that
\begin{align*}
\langle c_{k}, \varphi_i \rangle &=  \delta_{k,i} k_i
&
\langle c_{k,l}, \varphi_{i,j} \rangle &= \delta_{i,k} \delta_{j,l} k_{i,j}
&
\langle c_{1,2,3}, \varphi_{1,2,3} \rangle &= c_{1,2,3},
\end{align*}
for some nonzero constants $k_{I} \in \C$, where the Hochschild cycles are:
\begin{itemize}
\item 
Degree 1: $c_1 = U_1^* \otimes U_1$, $c_2 = U_2^* \otimes U_2$ and $c_3  = \xi_1^* \otimes  \xi_1 + \xi_2^* \otimes \xi_2$.
\item 
Degree 2: ``skewsymmetrisation'' $c_{j,3}$ of $\sum_{p} \xi_p^* U_j^* \otimes U_j \otimes \xi_p $:
$$ 
c_{j,3} = \sum_{p=1}^2 \xi_p^* U_j^* \otimes U_j \otimes \xi_p - U_j^* \xi_p^* \otimes \xi_p \otimes U_j.
$$ 
\item 
Degree 3: setting $U_{p,j}$ equal to $U_j$ for $j = 1,2$ and equal to $\xi_p$ for $j = 3$,
$$ 
c_{1,2,3} = \sum_{p =1}^2 \sum_{\sigma \in \Sigma_3} \varepsilon(\sigma) U_{p,\sigma(3)}^* U_{p,\sigma(2)}^* U_{p,\sigma(1)}^* \otimes  U_{p,\sigma(1)} \otimes  U_{p,\sigma(2)} \otimes  U_{p,\sigma(3)}.
$$
\end{itemize}
\end{prop}

\begin{remarque}
The Hochschild cycle $c_{1,2,3}$ is the analog of a fundamental form for $\QHMl$.
\end{remarque}

\begin{remarque}
Cyclic cocycles are Hochschild cocycles and thus the pairings $\langle c_I, \varphi_K \rangle $ between Hochschild cycles and cocycles exist. If we consider only \emph{Hochschild} cocycles, $\varphi_{1,3}$, $\varphi_{2,3}$ and $\varphi_{1,2}$ are dual in the above sense to $c_{1,3}$, $c_{2,3}$ and $c_{1,2}$ where $c_{1,2} = U_1^* U_2^* \otimes U_2 \otimes U_1 - U_2^* U_1^* \otimes U_1 \otimes U_2$.
\end{remarque}

\begin{remarque}
The formulas of these cycles are very similar to shuffle products of $c_i$. However, shuffle products don't apply directly here, since $\QHMl$ is not commutative.
\end{remarque}

\begin{proof}
It is obvious that $c_1$ and $c_2$ are Hochschild cycles. It is a straightforward consequence of \eqref{Eqn:Frame} that $c_3$ is closed. 

To prove that the $c_{i,3}$ and $c_{1,2,3}$ are Hochschild cycles, we essentially adapt the proof of \cite{EltNCG}, lemma 12.15. We can apply the same arguments because
\begin{align*}
e^{- i 4 \pi \mu} U_1 \xi &= \xi U_1 & e^{- i 4 \pi \nu} U_2 \xi &= \xi U_2.
\end{align*}
To compensate the fact that $\xi_p$ is not unitary, we sum over $p$ and use \eqref{Eqn:Frame}. The proof is then a set of lengthy but otherwise straightforward computations.

We only prove that $c_{1,2,3}$ is a Hochschild cycle and that it pairs non trivially with $\varphi_{1,2,3}$, the other calculations are easier cases of the same thing. We want to evaluate
\begin{multline*}
\left \langle \phi_{1,2,3} , \sum_{i} v_i \right \rangle = \\
\sum_{i} \sum_{\sigma \in \Sigma_3} \sum_{\sigma' \in \Sigma_3} \varepsilon(\sigma') \varepsilon(\sigma) \tau( U_{i,\sigma(3)}^* U_{i,\sigma(2)}^* U_{i,\sigma(1)}^* \partial_{\sigma'(1)} U_{i,\sigma(1)} \partial_{\sigma'(2)} U_{i,\sigma(2)} \partial_{\sigma'(3)} U_{i,\sigma(3)}).
\end{multline*}
Given $\sigma$ and $\sigma'$, we denote by $T_{\sigma,\sigma'}$ the associated term in the above sum. Notice that $\partial _3(U_{p,j}) = \delta_{3,j} i 2 \pi U_{p,j}$. We can find $j_3$ such that $\sigma(j_3) = 3$. If $\sigma(j_3) = 3 \neq \sigma'(j_3)$, then $T_{\sigma,\sigma'} = 0$. As $\partial _2(U_{p,j}) = - \delta_{2,j} i 2 \pi U_{p,j}$, if $\sigma(j_2) = 2 \neq \sigma'(j_2)$ then $T_{\sigma, \sigma'} = 0$. Thus $\sigma \neq \sigma'$ implies $T_{\sigma, \sigma'} = 0$. Finally, $\partial _1(U_{p,j}) = -\delta_{1,j} i 2 \pi U_{p,j}$ and 
\begin{multline*}
\left \langle \phi_{1,2,3} , \sum_{i} v_i \right \rangle =(i 2 \pi)^3 \sum_{i} \sum_{\sigma \in \Sigma_3} \tau \Big( U_{i,\sigma(3)}^* U_{i,\sigma(2)}^* U_{i,\sigma(1)}^* U_{i,\sigma(1)} U_{i,\sigma(2)} U_{i,\sigma(3)} \Big)\\
= (i 2 \pi)^3 \sum_{i} \sum_{\sigma \in \Sigma_3} \tau(\xi_i^* \xi_i )=(i 2 \pi)^3 6 \neq 0,
\end{multline*}
by using the trace property of $\tau$.
\end{proof}

About the construction of $U_3$: the proposition \ref{Prop:RelUnit} and the lemma \ref{Lem:IsomPP} essentially show that $U_3$ realizes an isomorphism $E^{c} \oplus E^{-c} \simeq 2 E^0$, where $E^t$ is the line bundle over $T^2$ with twisting $t$. Another unitary could be constructed, that would realize the isomorphism $E^1 \oplus E^{-1} \simeq E^{c-1} \oplus E^{-(c-1)}$. This unitary together with $U_1, U_2$ would probably be a generator of $K_1(\QHM)$. However, the necessary computations are much more involved.

\medbreak

To conclude, on top of its possible implications on noncommutative 3-spheres, our study may foster intuition on Pimsner algebras. In the case of QHM, the ``transfer formula'' \eqref{Eqn:Trans} shows how we can ``transfer'' a pairing from the Pimsner algebra to the basis algebra. This property is very similar to proposition 12.6 of Nest's article \cite{CohomCyclZ}. It would be interesting to investigate whether this is a general phenomenon for Pimsner algebras.

	\bibliographystyle{alpha} 
	\bibliography{../../Biblio/biblioCourtL1}

Université Denis \textsc{Diderot} Paris 7 -- Institut de Mathématiques de Jussieu

 \'Equipe algèbres d'opérateurs, 175, rue du Chevaleret, F-75 013 Paris FRANCE

\textsl{E-mail address:} \texttt{gabriel@math.jussieu.fr}
\end{document}